%% file: mauro-temp-2.tex
\documentclass[a4paper, 10pt]{article}

\usepackage{amsfonts,enumerate}
\usepackage{amsmath}
\usepackage{enumerate}
\usepackage{graphicx}
\usepackage[colorlinks=true]{hyperref}

\newtheorem{proposition}{Proposition}[section]
\newtheorem{theorem}[proposition]{Theorem}
\newtheorem{lemma}[proposition]{Lemma}

\newenvironment{proof}{\smallskip\noindent\emph{Proof.}\hspace{1pt}}%
{\hspace{-5pt}{\nobreak\quad\nobreak\hfill\nobreak$\square$\vspace{8pt}%
\par}\smallskip\goodbreak}

\newcommand{\C}[1]{\mathbf{C^{#1}}}

\newcommand{\naturali}{{\mathbb{N}}}

\renewcommand{\epsilon}{\varepsilon}
\renewcommand{\phi}{\varphi}
\renewcommand{\theta}{\vartheta}

\renewcommand{\L}[1]{\mathbf{L^#1}}

\begin{document}

\title{A Riemann solver at a junction compatible with a homogenization limit}

\author{M. Garavello\thanks{Dipartimento di Matematica e Applicazioni,
    Universit\`a di Milano Bicocca, Via R. Cozzi 55, 20125 Milano, Italy.
    E-mail: \texttt{mauro.garavello@unimib.it}} \and
  F. Marcellini\thanks{Dipartimento di Matematica e Applicazioni,
    Universit\`a di Milano Bicocca, Via R. Cozzi 55, 20125 Milano, Italy.
    E-mail: \texttt{francesca.marcellini@unimib.it}}}

\maketitle

\begin{abstract}
  We consider a junction regulated by a traffic lights,
  with $n$ incoming roads and only one outgoing road.
  On each road the Phase Transition traffic model,
  proposed in~\cite{ColomboMarcelliniRascle}, describes the evolution of
  car traffic. Such model is an extension of the classic
  Lighthill-Whitham-Richards one, obtained by assuming that 
  different drivers may have different maximal speed.

  By sending to infinity the number of cycles of the traffic lights,
  we obtain a justification of the Riemann solver introduced 
  in~\cite{GaravelloMarcellini} and in particular of the rule
  for determining the
  maximal speed in the outgoing road.

\noindent\textit{2000~Mathematics Subject Classification:} 35L65,
  90B20

  \medskip

\end{abstract}

\section{Introduction}
\label{sec:Intro}

This paper deals with the Phase Transition traffic model,
proposed by Colombo, Marcellini, and Rascle in~\cite{ColomboMarcelliniRascle},
at a junction with $n \ge 2$ incoming roads and only one outgoing road.
The aim is to give a mathematical derivation for the solution
of the Riemann problem at the crossroad 
proposed in~\cite{GaravelloMarcellini}. We obtain the justification
for such solution by using a homogenization procedure.

The traffic model considered in this paper is a system of $2 \times 2$
conservation laws; it belongs to the class of macroscopic second
order models as the famous Aw-Rascle-Zhang model,
see~\cite{AwRascle, Zhang2002}. As the name Phase Transition suggests,
the model is characterized by two different phases,
the free one and the congested one;
see~\cite{BlandinWorkGoatinPiccoliBayen, Colombo1.5,
  MR2032809, 2017arXiv170203624F, GaravelloHanPiccoli,
  Goatin2Phases, LebacqueMammarHajSalem, Marcellini, Marcellini2} and the references
therein for similar descriptions. 
The Phase Transition model we consider here is derived from the famous
Lighthill-Whitham-Richards one~\cite{LighthillWhitham, Richards}
by assuming that different drivers may have different maximal speeds.

The extension of the Phase Transition model to the case of networks
is considered in~\cite{GaravelloMarcellini}. 
The key point for extending a model to a network
consists in providing a concept of solution at nodes.
A possible way to do this is to construct a Riemann solver at nodes,
i.e. a function
which associates to each Riemann problem at the node a solution.
A reasonable Riemann solver has to satisfy the mass conservation,
a consistency condition; it should produce waves with
negative speed in the incoming edges and with
positive speed in the outgoing ones.
A Riemann solver satisfying
such properties is proposed in~\cite{GaravelloMarcellini}.
In particular, it prescribes that
the maximal speed in the outgoing road is a convex combination of the
maximal speed in the incoming arcs. Similar conditions are also present
in~\cite{HertyKlar2003, HertyRascle, 2017arXiv170701683K}.

In this paper we are going to investigate the delicate issue of
how the maximal speed \emph{changes} through the junction.
To this aim, we consider a single junction regulated by a time-periodic
traffic lights. At each time the green light applies only at one incoming
road. Vehicles, in the remaining incoming roads, are then stopped, waiting
for their green light.
With a limit-average procedure, we are able to find the relation between
the incoming maximal speeds and the outgoing one.
In this way, the maximal outgoing speed turns out to be a
convex combination of the $n$ incoming ones and it satisfies
the corresponding condition prescribed by the Riemann solver 
in~\cite{GaravelloMarcellini}.

The paper is organized as follow. In the next section we recall the 2-Phases
Traffic Model introduced in~\cite{ColomboMarcelliniRascle} and the solution
to the classical
Riemann problem along a single road of infinite length.
In Section~\ref{sec:Main} we consider a time periodic
traffic lights regulating the intersection and we study the solution in the
outgoing road as the time period of the traffic lights tends to $0$.
More precisely, in Subsection~\ref{subsec:I} we describe in details
the solution in the simple situation
with $n = 2$ incoming roads and, finally, in Subsection~\ref{subsec:II}
we generalize the previous study to the case of $n \ge 2$ incoming roads
and we state and prove the main result, concerning 
the rule for the maximal speed
in the outgoing road,
by using a limit-average procedure.

\section{Notations and the Riemann Problem on a Single Road}
\label{sec:Des}
The Phase Transition model, introduced in~\cite{ColomboMarcelliniRascle}, 
is given by: 
\begin{equation}
  \label{eq:Modeleta}
  \left\{
    \begin{array}{l}
      \partial_t \rho +
      \partial_x \left( \rho\, v (\rho,\eta) \right) = 0
      \\
      \partial_t \eta +
      \partial_x \left( \eta\, v (\rho, \eta) \right) = 0
    \end{array}
  \right.
  \quad \mbox{ with } \quad
  v(\rho, \eta)
  =
  \min \left\{ V_{\max}, \frac{\eta}{\rho}\, \psi(\rho) \right\},
\end{equation}
where $t$ denotes the time, $x$ the space, $\rho \in [0, R]$
is the traffic density, $\eta$ is a generalized
momentum, $v \in [0, V_{\max}]$ is the speed of cars, and
$V_{\max}$ is a uniform bound of the cars' speed.

It is obtained as an extension of the Lighthill-Whitham-Richards
model~\cite{LighthillWhitham,Richards}, by assuming that different
drivers have different maximal speed, denoted by the quantity
$w=\eta / \rho \in \left[\check w, \hat w\right]$.
It is characterized by two phases,
the free one and congested one, which are described by the sets
\begin{align}
  \label{eq:phF}
  F
  & = 
    \left\{
    (\rho, \eta) \in [0,R] \times [0, \hat w R]
    \colon \check w \rho \le \eta \le \hat w \rho, \,
    v(\rho, \eta) = V_{\max}
    \right\},
  \\
  \label{eq:phC}
  C
  & =
    \left\{
    (\rho, \eta) \in [0,R] \times [0, \hat w R]
    \colon 
    \check w \rho \le \eta \le \hat w \rho, \,
    v(\rho, \eta) = \frac{\eta}{\rho} \, \psi(\rho)
    \right\}\,,
\end{align}
see Figure~\ref{fig:phases}.
\begin{figure}[t!]
\centering
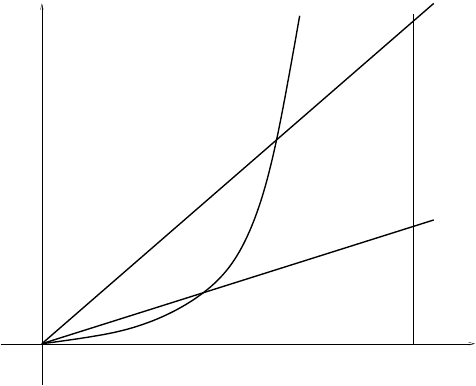
\hfil
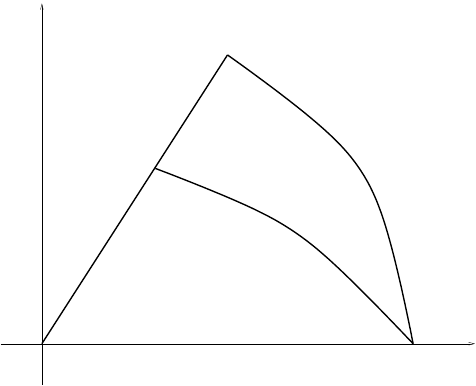
\caption{The free phase $F$ and the congested phase $C$ resulting
from\/ {\rm(\ref{eq:Modeleta})} in the coordinates, from left to right, $(\rho,\eta)$ and $(\rho, \rho v)$. Note that F and C are closed sets and $F \cap C \neq \emptyset $. Note also that $F$ is $1$--dimensional in the $(\rho, \rho v)$ plane, while it is $2$--dimensional in the $(\rho,\eta)$ coordinates. In the $(\rho,\eta)$ plane, the curve
    $\eta= \frac{V_{\max}}{\psi(\rho)}\rho $ divides the two phases.}\label{fig:phases}
\end{figure} 
As in~\cite{ColomboMarcelliniRascle,GaravelloMarcellini}, we assume the following hypotheses.
\begin{description}
\item[(H-1)] $R,\check w, \hat w, V_{\max}$ are positive
  constants, with $V_{\max} < \check w < \hat w$.

\item[(H-2)] $\psi \in \C{2} \left( [0,R];[0,1]\right)$ is such that
  $\psi(0) = 1$, $\psi(R) = 0$, and, for every $\rho \in (0,R)$,
  $\psi'(\rho) \le 0$,
  $\frac{d^2\ }{d\rho^2} \left( \rho\, \psi(\rho) \right) \le 0$.

\item[(H-3)] Waves of the first family in the congested phase $C$ have negative speed.
\end{description}

By~\textbf{(H-1)}, \textbf{(H-2)}, and \textbf{(H-3)}, system~(\ref{eq:Modeleta})
is strictly hyperbolic in $C$, see~\cite{ColomboMarcelliniRascle}, and
\begin{equation*}
  \begin{array}{@{}rcl@{\quad}rcl@{}}
    \lambda_{1} (\rho, \eta)
    \!\!&\!\! = \!\!&\!\!
                      \eta\, \psi'(\rho) + v(\rho, \eta),
    &
      \lambda_{2} (\rho, \eta)
      \!\!&\!\! = \!\!&\!\!
                        v(\rho, \eta),
    \\[5pt]
    r_{1} (\rho, \eta)
    \!\!&\!\! = \!\!&\!\!
                      \left[
                      \begin{array}{c}
                        -\rho
                        \\
                        -\eta
                      \end{array}
    \right],
        &
          r_{2} (\rho, \eta)
          \!\!&\!\! = \!\!&\!\!
                            \left[
                            \begin{array}{c}
                              1
                              \\
                              \eta\left( \frac{1}{\rho}-\frac{\psi'(\rho) }
                              {\psi(\rho) }\right)
                            \end{array}
    \right],
    \\
    \nabla \lambda_1 \cdot r_1
    \!\!&\!\! = \!\!&\!\!
                      \displaystyle
                      -\frac{d^2\ }{d\rho^2} \left[ \rho\, \psi(\rho) \right],
    &
      \nabla \lambda_2 \cdot r_2
      \!\!&\!\! = \!\!&\!\!
                        0,
    \\
    \mathcal{L}_1(\rho;\rho_o,\eta_o)
    \!\!&\!\! = \!\!&\!\!
                      \displaystyle
                      \eta_o \frac{\rho}{\rho_o},
    &
      \mathcal{L}_2(\rho;\rho_o,\eta_o)
      \!\!&\!\! = \!\!&\!\!
                        \displaystyle
                        \frac{\rho \, v(\rho_o, \eta_o)}{\psi(\rho)},
                        \; \rho_o < R,
  \end{array}
\end{equation*}
where $\lambda_i$ and $r_i$ are respectively the eigenvalues and
the right eigenvectors of the Jacobian matrix of the flux, and 
$\mathcal L_i$ are the Lax curves. 
When $\rho_o = R$, the 2-Lax curve through $(\rho_o, \eta_o)$ is given by
the segment $\rho=R$, $\eta \in [R \check w, R \hat w]$.  

In view of the results of the next section, we recall the description of
the solutions of the Riemann problem for the model~(\ref{eq:Modeleta}).
First, we enumerate all the possible waves for~(\ref{eq:Modeleta}).
\begin{itemize}
\item \textsl{A Linear wave} is a wave connecting two states in the free
  phase. It always travels with speed $V_{\max}$.

\item \textsl{A Phase Transition Wave} is a wave connecting a left state
  $\left(\rho_l, \eta_l\right) \in F$ with a right state
  $\left(\rho_r, \eta_r\right) \in C$ satisfying 
  $\frac{\eta_l}{\rho_l} = \frac{\eta_r}{\rho_r}$.
  It always travels with speed given by the Rankine-Hugoniot condition.

\item \textsl{A Wave of the First Family} is a wave connecting a left state 
  $\left(\rho_l, \eta_l\right) \in C$ with a right state
  $\left(\rho_r, \eta_r\right) \in C$ such that
  $\frac{\eta_l}{\rho_l} = \frac{\eta_r}{\rho_r}$. It is either a rarefaction
  wave or a shock wave.

\item \textsl{A Wave of the Second Family} is a wave connecting a left state 
  $\left(\rho_l, \eta_l\right) \in C$ with a right state
  $\left(\rho_r, \eta_r\right) \in C$ such that
  $v\left(\rho_l, \eta_l\right) = v\left(\rho_r, \eta_r\right)$.
  It always travels with speed $v\left(\rho_l, \eta_l\right)$.
\end{itemize}

\subsection{The Riemann Problem along a Single Road}
Under the assumptions \textbf{(H-1)}, \textbf{(H-2)} and \textbf{(H-3)},
for all states $(\rho_l,\eta_l)$ and $(\rho_r, \eta_r) \in F \cup C$,
the Riemann problem consisting of~(\ref{eq:Modeleta}) with initial data
\begin{equation}
  \label{eq:RD}
  \rho(0,x) = \left\{
    \begin{array}{l@{\quad\mbox{ if }\,}rcl}
      \rho_l & x & < & 0
      \\
      \rho_r & x & > & 0
    \end{array}
  \right.
  \qquad
  \eta(0,x) = \left\{
    \begin{array}{l@{\quad\mbox{ if }\, }rcl}
      \eta_l & x & < & 0
      \\
      \eta_r & x & > & 0
    \end{array}
  \right.
\end{equation}
admits a unique self similar weak solution $(\rho,\eta) =
(\rho,\eta) (t,x)$ constructed as follows:
\begin{enumerate}[(1)]
\item If $(\rho_l,\eta_l), (\rho_r,\eta_r) \in F$, then the solution attains values in $F$ and consists of a linear wave separating $(\rho_l,\eta_l)$ from $(\rho_r,\eta_r)$.  
  
\item If $(\rho_l,\eta_l), (\rho_r,\eta_r) \in C$, then the solution attains values in $C$ and consists of a wave of the first family (shock or rarefaction) between $(\rho_l, \eta_l)$ and a middle state $(\rho_m, \eta_m)$, followed by a wave of the second family between $(\rho_m, \eta_m)$ and $(\rho_r, \eta_r)$.  The middle state $(\rho_m, \eta_m)$ belongs to $C$ and is uniquely characterized by the two conditions $\frac{\eta_m}{\rho_m} = \frac{\eta_l}{\rho_l}$ and
  $v(\rho_m, \eta_m) = v(\rho_r, \eta_r)$.

\item If $(\rho_l,\eta_l) \in C$ and $(\rho_r,\eta_r) \in F$, then the solution attains values in $F \cup C$ and consists of a wave of the first family separating $(\rho_l, \eta_l)$ from a middle state $(\rho_m, \eta_m)$ and by a linear wave separating $(\rho_m, \eta_m)$ from $(\rho_r,\eta_r)$. The middle state $(\rho_m, \eta_m)$ belongs to the intersection between $F$ and $C$ and is uniquely characterized by the two conditions $\frac{\eta_m}{\rho_m} = \frac{\eta_r}{\rho_r}$ and $v(\rho_m, \eta_m) = V_{\max}$.

  \item If $(\rho_l,\eta_l) \in F$ and $(\rho_r,\eta_r) \in C$, then the solution attains values in $F \cup C$ and consists of a phase transition wave between $(\rho_l, \eta_l)$ and a middle state $(\rho_m, \eta_m)$, followed by a wave of the second family between $(\rho_m, \eta_m)$ and $(\rho_r, \eta_r)$.  The middle state $(\rho_m, \eta_m)$ is in $C$ and is uniquely characterized by the two conditions $\frac{\eta_m}{\rho_m} = \frac{\eta_l}{\rho_l}$ and $v(\rho_m, \eta_m) = v(\rho_r, \eta_r)$.
 
\end{enumerate}  

\section{The Limit at a Junction with Traffic Lights}
\label{sec:Main}

Fix a junction with $n$ incoming roads and a single outgoing road.
In~\cite{GaravelloMarcellini}, it is introduced a Riemann solver at the
junction, which conserves the mass 
\begin{equation}
  \label{eq:conservation1}
  \sum_{i = 1}^n \rho_i v(\rho_i, \eta_i) = 
  \rho_{n+1} v\left(\rho_{n+1}, \eta_{n+1}\right)\,,
\end{equation}
and prescribes that the maximal speed in the outgoing road is given by
\begin{equation}
  \label{eq:conservation2}
  w_{n+1} = \frac{\displaystyle\sum_{i=1}^n \sigma_i \rho_i
    v(\rho_i, \eta_i) \,w_i}
  {\displaystyle\sum_{i=1}^n \sigma_i \rho_i v(\rho_i, \eta_i) },
\end{equation}
for suitable coefficients $\sigma_1 > 0, \cdots, \sigma_n > 0$, satisfying
$\sigma_1 + \cdots + \sigma_n = 1$.

In this section we provide a justification for the
rule~(\ref{eq:conservation2}).
To this aim, fix a positive time $T > 0$ and assume that the junction is
regulated by a traffic lights, which alternates periodically the right of
way among the incoming roads. More precisely, assume that the traffic lights
repeats $\ell \in \naturali \setminus \left\{0\right\}$ cycles in the time
interval $[0,T]$ and that each cycle of length $\frac{T}{\ell}$ is divided
into $n$ subintervals of length $\tau_1^\ell, \cdots, \tau_n^\ell$,
which represent respectively
the duration of the green light for the corresponding
incoming road. The first cycle $\left[0, \frac{T}{\ell}\right[$
is thus composed by
\begin{equation*}
  \left[0, \frac{T}{\ell}\right[ = 
  \left[0, \tau_1^\ell\right) \bigcup
  \left[\tau_1^\ell, \tau_1^\ell + \tau_2^\ell\right)
  \bigcup \cdots \bigcup
  \left[\tau_1^\ell + \cdots + \tau_{n-1}^\ell, \tau_1^\ell + \cdots 
    + \tau_n^\ell\right),
\end{equation*}
where $\left[0, \tau_1^\ell\right)$ is the time interval of the green for
the road $I_1$ and so on.

Denote, for every $i \in \left\{1, \ldots, n-1\right\}$,
\begin{equation}
  \label{eq:sigma}
  \sigma_i = \frac{\tau_i^\ell}{\tau_n^\ell},
\end{equation}
which we suppose that it does not depend on $\ell$. Thus the constant
$\sigma_i$ is the ratio between the green time
interval for roads $I_i$ and $I_n$. For simplicity we put $\sigma_n = 1$.

\subsection{Basic Situations}
\label{subsec:I}
In this subsection we treat only the special junction with $n = 2$ incoming
roads (namely $I_1$ and $I_2$) and a single outgoing road $I_3$.
We assume that 
assumptions~\textbf{\textup{(H-1)}}, \textbf{\textup{(H-2)}},
and~\textbf{\textup{(H-3)}} hold. As introduced in~(\ref{eq:sigma}), the
positive constants
\begin{equation*}
  \sigma_1 = \frac{\tau_1^\ell}{\tau_2^\ell},\qquad
  \qquad\sigma_2 = 1
\end{equation*}
do not depend on the number of cycles $\ell$.
This means that the ratio between the green and red times is constant
in each incoming road.

Given, for every $i \in \left\{1, 2, 3\right\}$, initial conditions
$(\bar \rho_{i}, \bar \eta_{i}) \in F \cup C$, we denote
with $\left(\rho_{\ell,i} (t,x), \eta_{\ell,i} (t,x)\right)$
($i \in \left\{1, 2, 3\right\}$)
the solution to the Riemann problem, the junction being governed by
the traffic lights with $\ell$ cycles.

 Introduce the following notation. With
  $\left(\rho_1^\sharp, \eta_1^\sharp\right)$ and
  $\left(\rho_2^\sharp, \eta_2^\sharp\right)$
  we call the points in the congested region $C$ satisfying
  \begin{equation*}
    \frac{\eta_1^\sharp}{\rho_1^\sharp} = \bar w_1,
    \qquad
    \frac{\eta_2^\sharp}{\rho_2^\sharp} = \bar w_2,
    \qquad
    v\left(\rho_1^\sharp, \eta_1^\sharp\right) =
    v\left(\rho_2^\sharp, \eta_2^\sharp\right) =
    v\left(\bar \rho_3, \bar \eta_3\right).
  \end{equation*}
  Moreover with
  $\left(\rho_1^\flat, \eta_1^\flat\right)$ and
  $\left(\rho_2^\flat, \eta_2^\flat\right)$
  we call the points in the intersection between the free and congested region $F \cap C$ satisfying
  \begin{equation*}
    \frac{\eta_1^\flat}{\rho_1^\flat} = \bar w_1,
    \qquad
    \frac{\eta_2^\flat}{\rho_2^\flat} = \bar w_2,
    \qquad
    v\left(\rho_1^\flat, \eta_1^\flat\right) =
    v\left(\rho_2^\flat, \eta_2^\flat\right) = V_{\max} .
  \end{equation*}
  Note that the points $\left(\rho_i^\sharp, \eta_i^\sharp\right)$ and
  $\left(\rho_i^\flat, \eta_i^\flat\right)$, $i = 1,2$,
  are uniquely defined; see Figure~\ref{fig:ex}.
\begin{figure}[t!]
    \centering
    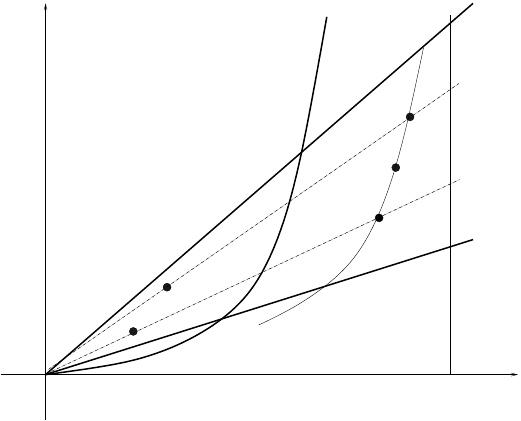
    \hfil
    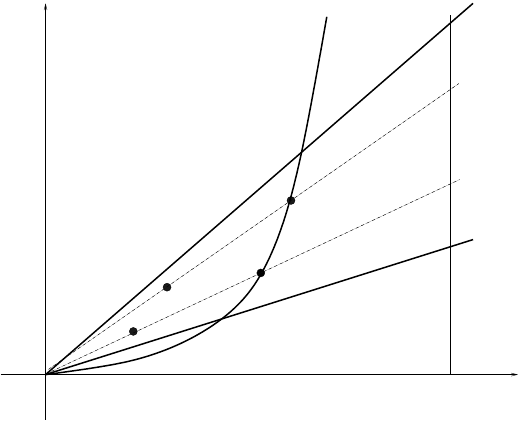
    \caption{Left, an example of representation for the points $\left(\rho_1^\sharp, \eta_1^\sharp\right)$ and $\left(\rho_2^\sharp, \eta_2^\sharp\right)$, right an example of representation for the points $\left(\rho_1^\flat, \eta_1^\flat\right)$ and
  $\left(\rho_2^\flat, \eta_2^\flat\right)$.}
\label{fig:ex}
  \end{figure}  

The next lemmas describe the solution at the junction with the traffic lights
for the possible different situations that may happen.
\begin{lemma}
\label{Lemma1}
  Assume that the initial conditions belong to the congested phase, i.e.,
  for every $i \in \left\{1, 2, 3\right\}$,
  $(\bar \rho_{i}, \bar \eta_{i}) \in C$.
  Then the solution in the outgoing road $I_3$ is
  \begin{equation}
    \label{eq:sol_I_3_case1}
    \left(\rho_{\ell, 3}, \eta_{\ell, 3}\right) (t,x) =
    \left\{
      \begin{array}{l@{\quad}l}
        \left(\bar \rho_3, \bar \eta_3\right),
        &
          \textrm{ if }\, 0 < t < T, \quad
          x > v\left(\bar \rho_3, \bar \eta_3\right) t,
        \\
        \left(\rho_{1}^\sharp, \eta_1^\sharp\right),
        &
          \textrm{ if }\, (t, x) \in A_1^\ell,
        \\
        \left(\rho_{2}^\sharp, \eta_2^\sharp\right),
        &
          \textrm{ if }\, (t, x) \in A_2^\ell,
      \end{array}
    \right.
  \end{equation}
  where
  \begin{equation}
    \label{eq:A_1}
    A_1^\ell = \bigcup_{i=0}^{\ell - 1} \left\{(t, x):\,
      \begin{array}{c}
        0 < t < T, \quad x > 0 
        \\
        t - \tau_1^\ell - i \frac{T}{\ell} 
        < \frac{x}{v\left(\bar \rho_3, \bar \eta_3\right)}
        < t - i \frac{T}{\ell}
      \end{array}
    \right\}
  \end{equation}
  and
  \begin{equation}
    \label{eq:A_2}
    A_2^\ell = \bigcup_{i=1}^{\ell} \left\{(t, x):\,
      \begin{array}{c}
        0 < t < T, \quad x > 0 
        \\
        t - i \frac{T}{\ell} 
        < \frac{x}{v\left(\bar \rho_3, \bar \eta_3\right)}
        < t - i \frac{T}{\ell} + \tau_1^\ell
      \end{array}
    \right\}.
  \end{equation}
\end{lemma}

\begin{proof}
  In the time interval $[0,\tau_{1}^\ell[$ the traffic lights is green 
  in the first incoming road; this permits to study the Riemann problem
  as a classical one considering a unique road given
  by the union of the $I_1$ and $I_3$:
  the classic Riemann problem between $(\bar \rho_{1}, \bar \eta_{1})$
  and $\left(\bar \rho_3, \bar \eta_3\right)$
  produces a first family wave between $(\bar \rho_{1}, \bar \eta_{1})$
  and $\left(\rho_1^\sharp, \eta_1^\sharp\right)$ and a second family wave between $\left(\rho_1^\sharp, \eta_1^\sharp\right)$
  and $\left(\bar \rho_3, \bar \eta_3\right)$, see Figure~\ref{fig:tau1}. 
  In $I_2$, instead, the flow at the junction is equal to zero and
  so the trace of the solution is $\left(R,R \bar w_2\right)$.
  The solution in the road $I_2$ is given by
  a shock wave of the first family connecting
  $\left(\bar \rho_2, \bar \eta_2\right)$ to $\left(R,R \bar w_2\right)$;
  see Figure~\ref{fig:tau11}.
  \begin{figure}[t!]
    \centering
    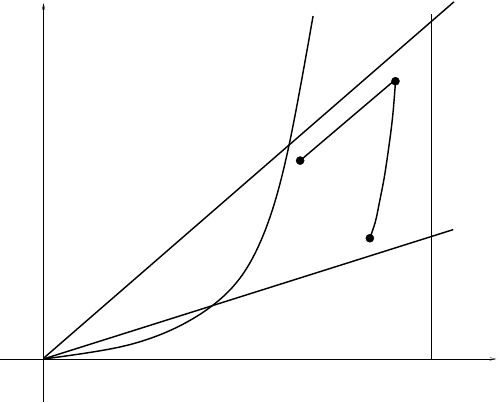
    \hfil
    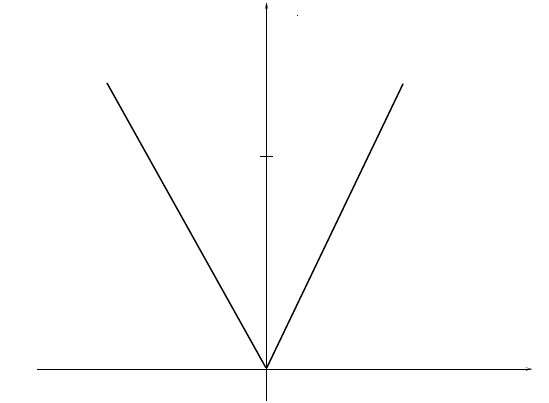
    \caption{The situation of Lemma~\ref{Lemma1} in the time interval $[0,\tau_{1}^\ell[$ for the first incoming road and the
      outgoing road in the coordinates, from left to right, $(\rho,\eta)$ and
      $(x, t)$.}\label{fig:tau1}
  \end{figure}  
  \begin{figure}[t!]
    \centering 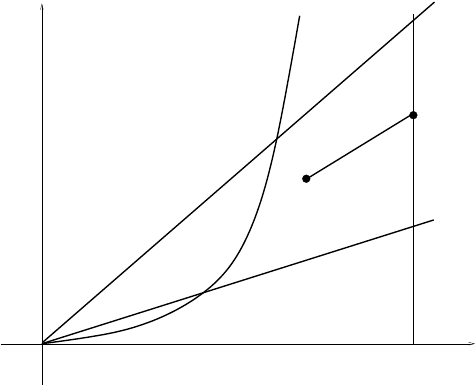 \hfil
    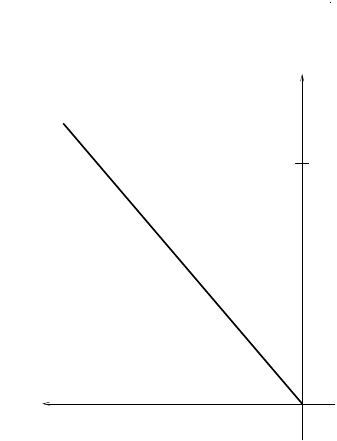
    \caption{The situation of Lemma~\ref{Lemma1} in the time interval $[0,\tau_{1}^\ell[$ for the second incoming road in
      the coordinates, from left to right, $(\rho,\eta)$ and
      $(x, t)$.}\label{fig:tau11}
  \end{figure}

  At time $t = \tau_1^\ell$, the traffic lights becomes red for the road $I_1$
  and green for $I_2$. This situation remains constant in the whole
  time interval $[\tau_{1}^\ell,\tau_{1}^\ell+\tau_{2}^\ell[$; this permits
  to study the Riemann problem
  as a classical one considering a unique road given
  by the union of the $I_2$ and $I_3$; see Figure~\ref{fig:tau2}.
  We have to solve the Riemann problem between
  $\left(R,R \bar w_2\right)$ and $\left(\rho_1^\sharp, \eta_1^\sharp\right)$.
  The solution is given by a rarefaction curve of the first family
  between $(R, R \bar w_2)$ and $\left(\rho_2^\sharp, \eta_2^\sharp\right)$
  followed by a second family wave between $\left(\rho_2^\sharp, \eta_2^\sharp\right)$ and $\left(\rho_1^\sharp, \eta_1^\sharp\right)$, see Figure~\ref{fig:tau2}. 
  In $I_1$, instead, the flow at the junction is equal to zero and
  so the trace of the solution is $\left(R,R \bar w_1\right)$.
  More precisely a shock wave of the first family starts from the point
  $\left(\tau_1^\ell, 0\right)$ connecting the states
  $\left(\rho_1^\sharp, \eta_1^\sharp\right)$ and $\left(R,R \bar w_1\right)$;
  see Figure~\ref{fig:tau22}.
  \begin{figure}[t!]
    \centering 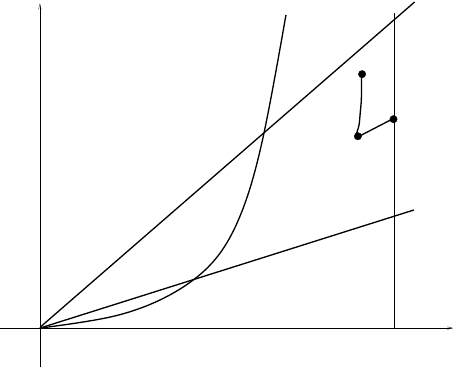 \hfil
    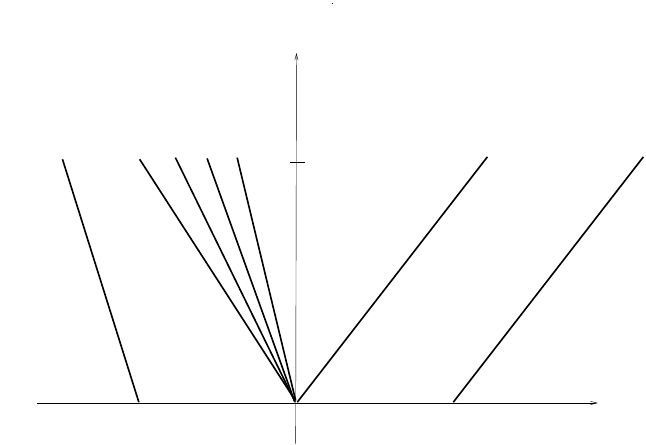
    \caption{The situation of Lemma~\ref{Lemma1} in the time interval $[\tau_{1}^\ell,\tau_{1}^\ell+\tau_{2}^\ell[$ for the second
      incoming road and the outgoing road in the coordinates, from
      left to right, $(\rho,\eta)$ and $(x, t)$.}\label{fig:tau2}
  \end{figure}
  \begin{figure}[t!]
    \centering 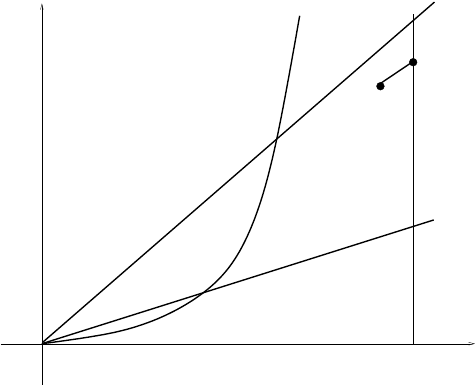 \hfil
    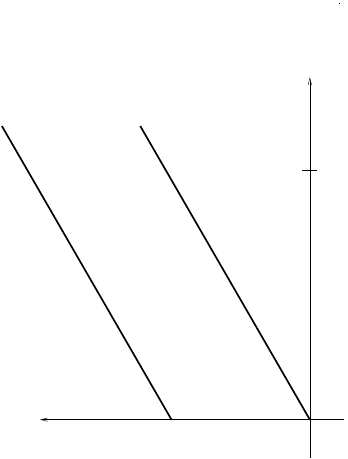
    \caption{The situation of Lemma~\ref{Lemma1} in the time interval $[\tau_{1}^\ell,\tau_{1}^\ell+\tau_{2}^\ell[$ for the first
      incoming road in the coordinates, from left to right,
      $(\rho,\eta)$ and $(x, t)$.}\label{fig:tau22}
  \end{figure}

  Similarly, in the time interval
  $[\tau_{1}^\ell + \tau_{2}^\ell, 2\tau_{1}^\ell + \tau_{2}^\ell[$,
  the traffic light is green for road $I_1$ and red for $I_2$;
  so we need to consider a Riemann
  problem between
  $(R,R \bar w_1)$ and $\left(\rho_2^\sharp, \eta_2^\sharp\right)$,
  see Figure~\ref{fig:tau3}. 
  The solution consists in a rarefaction curve of the first family
  between $(R,R \bar w_1)$ and $\left(\rho_1^\sharp, \eta_1^\sharp\right)$
  followed by a second family wave between
  $\left(\rho_1^\sharp, \eta_1^\sharp\right)$ and $\left(\rho_2^\sharp, \eta_2^\sharp\right)$.
  The situation of $I_2$ is analogous to that represented
  in Figure~\ref{fig:tau11}. More precisely at the point
  $\left(\tau_{1}^\ell + \tau_{2}^\ell, 0\right)$ a shock wave with negative
  speed is generated and it connects $\left(\rho_2^\sharp, \eta_2^\sharp\right)$
  with $\left(R,R \bar w_2\right)$.
  
  \begin{figure}[t!]
    \centering 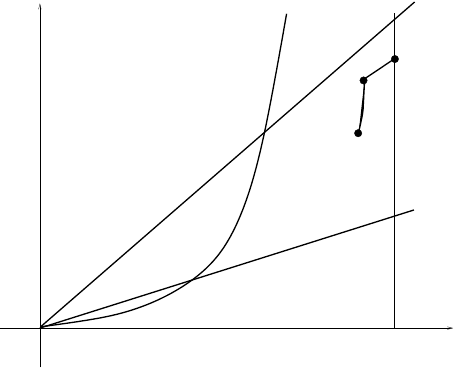 \hfil
    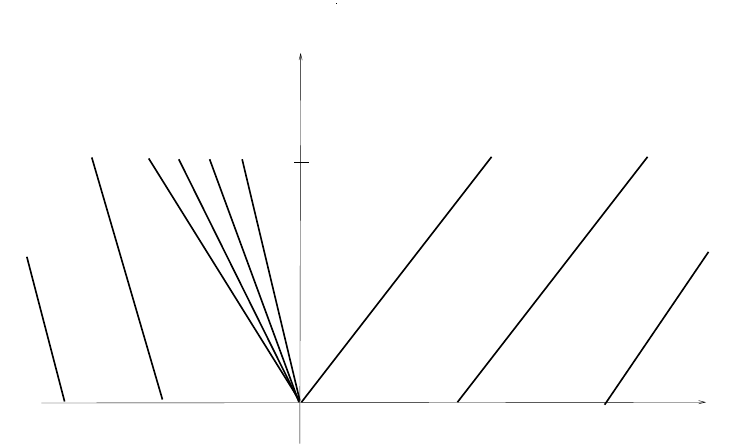
    \caption{The situation of Lemma~\ref{Lemma1} in the time interval $[\tau_{1}^\ell + \tau_{2}^\ell, 2\tau_{1}^\ell + \tau_{2}^\ell[$ for the first incoming road and the outgoing road in the coordinates,
      from left to right, $(\rho,\eta)$ and $(x, t)$.}\label{fig:tau3}
  \end{figure}
  \begin{figure}[t!]
    \centering 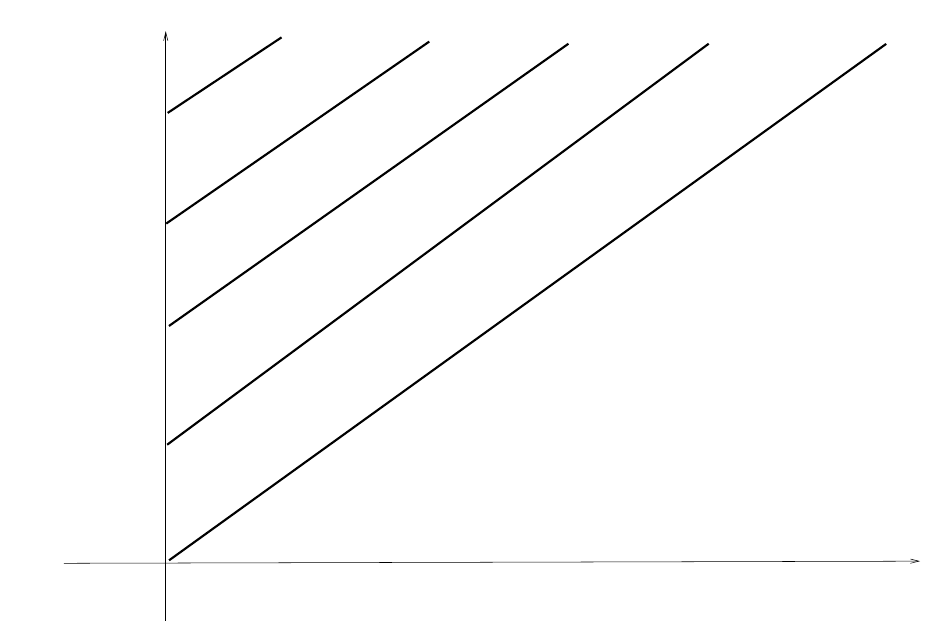
    \caption{The situation of Lemma~\ref{Lemma1} in the time interval $[0,2\tau_{1}^\ell+2\tau_{2}^\ell[$ for the outgoing road
      in the coordinates $(x, t)$. The two states $\left(\rho_1^\sharp, \eta_1^\sharp\right)$ and $\left(\rho_2^\sharp, \eta_2^\sharp\right)$, separated by second family waves, alternate periodically.}\label{fig:tau*}
  \end{figure}

  We proceed in the same way until we arrive at time $t = T$.
  In this way we deduce that
  the solution in $I_3$ is given by~(\ref{eq:sol_I_3_case1});
  see Figure~\ref{fig:tau*}. 
\end{proof}

\begin{lemma}
\label{Lemma2}
  Assume that the initial conditions satisfy
  \begin{equation*}
    \left(\bar \rho_{1}, \bar \eta_1\right) \in F, \quad
    \left(\bar \rho_{2}, \bar \eta_2\right) \in C, \quad
    \left(\bar \rho_{3}, \bar \eta_3\right) \in C.
  \end{equation*}
  Then either the solution in the outgoing road $I_3$ is given
  by~(\ref{eq:sol_I_3_case1}) or has a structure similar
  to~(\ref{eq:sol_I_3_case1})
  except for a set, whose Lebesgue measure is bounded by a constant times
  $\frac{1}{\ell^2}$.
\end{lemma}

\begin{proof}
  We proceed in the same way as the previous Lemma~\ref{Lemma1}.
  In the time interval $[0,\tau_{1}^\ell[$ the traffic lights is green 
  in the first incoming road and we study the Riemann problem
  as a classical one considering a unique road given
  by the union of the $I_1$ and $I_3$. We have two possible cases: the Riemann problem between $(\bar \rho_{1}, \bar \eta_{1})$ and $\left(\bar \rho_3, \bar \eta_3\right)$
  produces a shock wave with negative speed between $(\bar \rho_{1}, \bar \eta_{1})$
  and $\left(\rho_1^\sharp, \eta_1^\sharp\right)$ and a second family wave between $\left(\rho_1^\sharp, \eta_1^\sharp\right)$ and $\left(\bar \rho_3, \bar \eta_3\right)$, or produces a shock wave with positive speed between $(\bar \rho_{1}, \bar \eta_{1})$ and $\left(\rho_1^\sharp, \eta_1^\sharp\right)$ and a second family wave between $\left(\rho_1^\sharp, \eta_1^\sharp\right)$ and $\left(\bar \rho_3, \bar \eta_3\right)$, see Figure~\ref{fig:rhov}. 
  In $I_2$ the situation is the same as that of the previous Lemma~\ref{Lemma1}: the flow at the junction is equal to zero and the solution is given by a shock wave of the first family connecting $\left(\bar \rho_2, \bar \eta_2\right)$ to $\left(R,R \bar w_2\right)$, see Figure~\ref{fig:tau11}.
  \begin{figure}[t!]
    \centering
    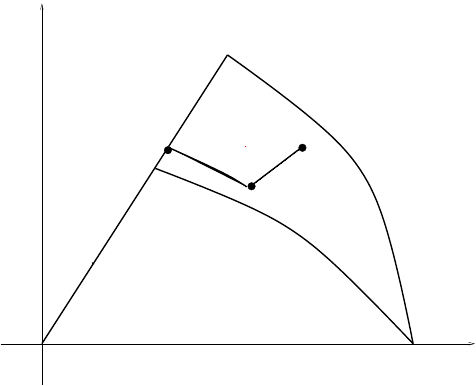
    \hfil
    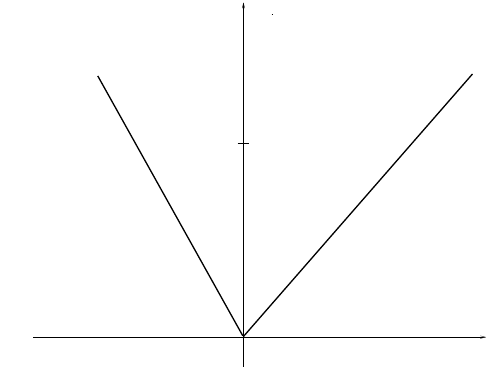
    \hfil 
    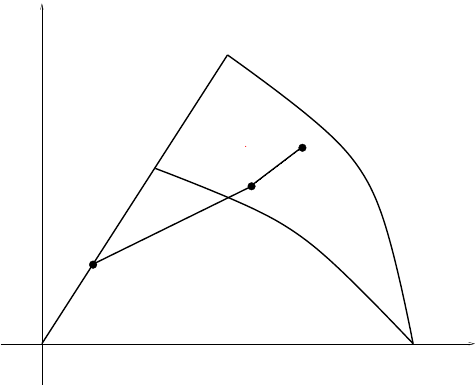
    \hfil
    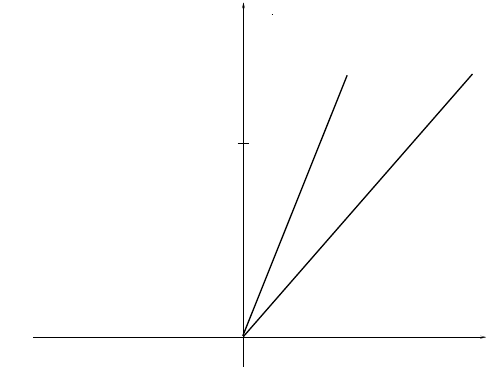
    \caption{The situation of Lemma~\ref{Lemma2} in the time interval $[0,\tau_{1}^\ell[$ for the first incoming road and the
      outgoing road in the coordinates, from left to right, $(\rho,\rho v)$ and
      $(x, t)$. Above, the first case, a shock with negative speed; below, the second case, a shock with positive speed.}\label{fig:rhov}
  \end{figure}  
\begin{itemize}
\item \textit{First case: a shock with negative speed.}
In this case, from time $\tau_{1}^\ell$ to time $t = T$, the solution becomes as that described in the previous Lemma~\ref{Lemma1} and it is given by~(\ref{eq:sol_I_3_case1}); see Figure~\ref{fig:tau*}.
  
\item \textit{Second case: a shock with positive speed.} In the time interval $[\tau_{1}^\ell,\tau_{1}^\ell+\tau_{2}^\ell[$ the traffic light is red for the road $I_1$
  and green for $I_2$ and we study the Riemann problem
  as a classical one considering a unique road given
  by the union of the $I_2$ and $I_3$. We have to solve the Riemann problem between
  $\left(R,R \bar w_2\right)$ and $(\bar \rho_{1}, \bar \eta_{1})$.
  The solution is given by a rarefaction curve of the first family
  between $(R, R \bar w_2)$ and
  $\left(\rho_2^\flat, \eta_2^\flat\right)\in F \cap C$
  followed by a linear wave wave between
  $\left(\rho_2^\flat, \eta_2^\flat \right)$ and
  $(\bar \rho_{1}, \bar \eta_{1})$. At time $t=\bar t_1$,
  the linear wave generated at the time $t=\tau_{1}^\ell$ between
  $\left(\rho_2^\flat, \eta_2^\flat\right)$
  and $(\bar \rho_{1}, \bar \eta_{1})$ interacts with the shock with positive speed genearated at time $t=0$ between $(\bar \rho_{1}, \bar \eta_{1})$ and $\left(\rho_1^\sharp, \eta_1^\sharp\right)$, see Figure~\ref{fig:part}. The intersection point is determined by solving the system:
  \begin{equation}
    \label{eq:Iintersection}
    \left\{
      \begin{array}{l}
        x(t)= v_s t
        \\
        x(t)=V_{\max}(t-\tau_{1}^\ell)\,,
      \end{array}
    \right.
  \end{equation}
where $v_s=\frac{\bar \rho_{1} V_{\max}- \rho_1^\sharp v(\rho_1^\sharp,\eta_1^\sharp) }{\bar \rho_{1} - \rho_1^\sharp}$.  We denote the intersection point by
  \begin{equation}
    \label{eq:PI}
    \left(\bar t_1,\bar x_1\right)
    = \left( \frac{V_{\max} \tau_{1}^\ell}{V_{\max}- v_s},
    v_s \frac{V_{\max} \tau_{1}^\ell}{V_{\max}- v_s} \right) \,.
  \end{equation}  
At this point we have to solve a Riemann problem
between $\left(\rho_2^\flat, \eta_2^\flat\right)$
and $\left(\rho_1^\sharp, \eta_1^\sharp\right)$ that generates a first family
wave between $\left(\rho_2^\flat, \eta_2^\flat\right)$
and $\left(\rho_2^\sharp, \eta_2^\sharp\right)\in C$
and a second family wave between $\left(\rho_2^\sharp, \eta_2^\sharp\right)$
and $\left(\rho_1^\sharp, \eta_1^\sharp\right)$, see Figure~\ref{fig:part}.
 \begin{figure}[t!]
    \centering
    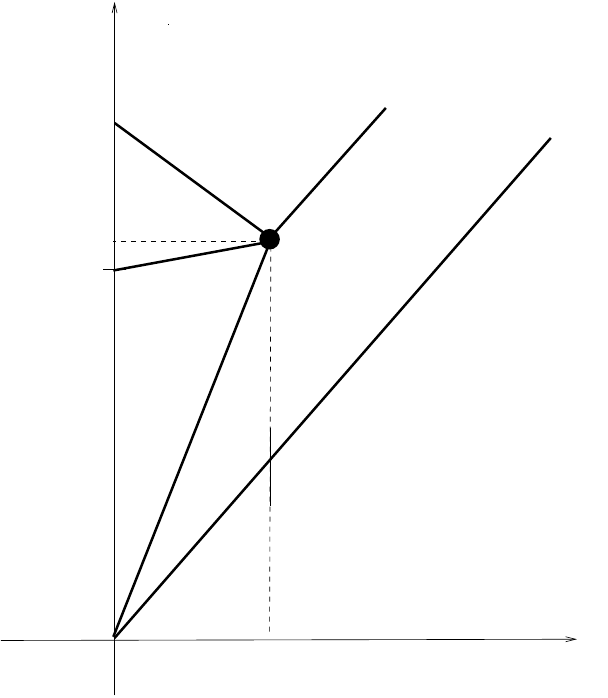
    \caption{The situation of Lemma~\ref{Lemma2} in the time interval  $[0,\bar t_2[$ for the outgoing road in the coordinates $(x, t)$: the interaction between the shock with positive speed genearated at time $t=0$ between $(\bar \rho_{1}, \bar \eta_{1})$ and $\left(\rho_1^\sharp, \eta_1^\sharp\right)$ and the linear wave generated at the time $t=\tau_{1}^\ell$ between $\left(\rho_m, \eta_m\right)$ and $(\bar \rho_{1}, \bar \eta_{1})$.}\label{fig:part}
  \end{figure}   

The first family wave between
$\left(\rho_2^\flat, \eta_2^\flat\right)$ and
$\left(\rho_2^\sharp, \eta_2^\sharp\right)$ could interact
again at a time $t=t_2$ with the linear wave generated at time
$t=\tau_{2}^\ell$ when the traffic lights is red for the road $I_2$
and green for $I_1$, producing again a first family wave and a second
family wave. 
A first family wave with negative speed could be produced at each interaction up to a time $t=t_*$, when it is absorbed and the solution becomes as that described in the previous Lemma~\ref{Lemma1} and it is given by~(\ref{eq:sol_I_3_case1}) except for a set, whose Lebesgue measure is bounded by a constant times $\frac{1}{\ell^2}$. 
Indeed, if for example we suppose that the first famiy wave
is absorbed at time $t_*=\bar t_2$ and, denoting with $\mathcal L$ the Lebesgue measure of a set, we estimate the area of ​​the triangle $A^\ell$ generated up to time $t=\bar t_2$, see Figure~\ref{fig:part}. By posing $\bar t_1=\frac{K_{1}}{\ell}$ and $\bar t_2= K_{2}\bar x_1$, we have:
  \begin{equation}
    \label{eq:AT}
    \mathcal L (A^\ell)=\frac{1}{2}v_s\frac{K_{1}}{\ell}\left(\frac{K_{1}}{\ell}+ K_{2}v_s\frac{K_{1}}{\ell}\right)  =\frac{K_{3}}{\ell^{2}} \,,
\end{equation}
\end{itemize}
for $K_{1},K_{2},K_{3}$ positive costants.
\end{proof}

\begin{lemma}
\label{Lemma3}
  Assume that the initial conditions satisfy
  \begin{equation*}
    \left(\bar \rho_{1}, \bar \eta_1\right) \in F, \quad
    \left(\bar \rho_{2}, \bar \eta_2\right) \in C, \quad
    \left(\bar \rho_{3}, \bar \eta_3\right) \in F.
  \end{equation*}
  Then the solution in the outgoing road $I_3$ is
  \begin{equation}
    \label{eq:I_3_Lemma3}
    \left(\rho_{\ell, 3}, \eta_{\ell, 3}\right) (t,x) =
    \left\{
      \begin{array}{l@{\quad}l}
      \left(\bar \rho_3, \bar \eta_3\right),
        &
          \textrm{ if }\, 0 < t < T, \quad x > V_{\max} t,
          \\
        \left(\bar \rho_1, \bar \eta_1\right),
        &
          \textrm{ if }
          \left\{          
          \begin{array}{l}
            0 < t < T,\, x < V_{\max} t,
            \\
            x > \max \left\{0, V_{\max} \left(t - \tau_1^\ell\right)\right\},
          \end{array}
        \right.        \\
        \left(\rho_{2}^\flat, \eta_2^\flat\right),
        &
          \textrm{ if }\, (t, x) \in A_1^\ell,
        \\
        \left(\rho_{1}^\flat, \eta_1^\flat\right),
        &
          \textrm{ if }\, (t, x) \in A_2^\ell,
      \end{array}
    \right.
  \end{equation}
  where
  \begin{equation}
    \label{eq:A_1}
    A_1^\ell = \bigcup_{i=0}^{\ell - 1} \left\{(t, x):\,
      \begin{array}{c}
        0 < t < T, \quad x > 0 
        \\
        t - \tau_1^\ell - i \frac{T}{\ell} 
        < \frac{x}{V_{max}}
        < t - i \frac{T}{\ell}
      \end{array}
    \right\}
  \end{equation}
  and
  \begin{equation}
    \label{eq:A_2}
    A_2^\ell = \bigcup_{i=1}^{\ell} \left\{(t, x):\,
      \begin{array}{c}
        0 < t < T, \quad x > 0 
        \\
        t - i \frac{T}{\ell} 
        < \frac{x}{V_{max}}
        < t - i \frac{T}{\ell} + \tau_1^\ell
      \end{array}
    \right\};
  \end{equation}

\end{lemma}

\begin{proof}
 We proceed in the same way of the previous lemmas.
 In the time interval $[0,\tau_{1}^\ell[$ the traffic lights is green 
 in the first incoming road and we study the Riemann problem
 as a classical one considering a unique road given
 by the union of the $I_1$ and $I_3$:
 the Riemann problem between $(\bar \rho_{1}, \bar \eta_{1})$
 and $\left(\bar \rho_3, \bar \eta_3\right)$
 produces a linear wave between $(\bar \rho_{1}, \bar \eta_{1})$
 and  $\left(\bar \rho_3, \bar \eta_3\right)$, see
 Figure~\ref{fig:tau1_Lemma3}. 
 In $I_2$, the flow at the junction is equal to zero and the solution in
 the road $I_2$ is given by a shock wave of the first family connecting
 $\left(\bar \rho_2, \bar \eta_2\right)$ to $\left(R,R \bar w_2\right)$,
 see Figure~\ref{fig:tau11_Lemma3}.
  \begin{figure}[t!]
    \centering
    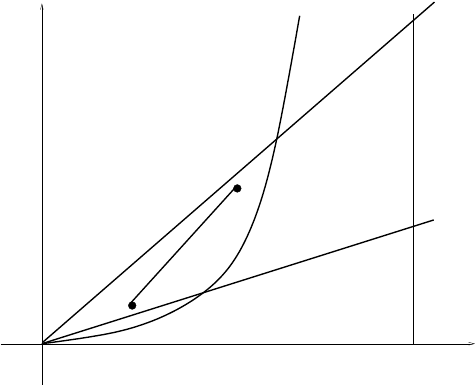
    \hfil
    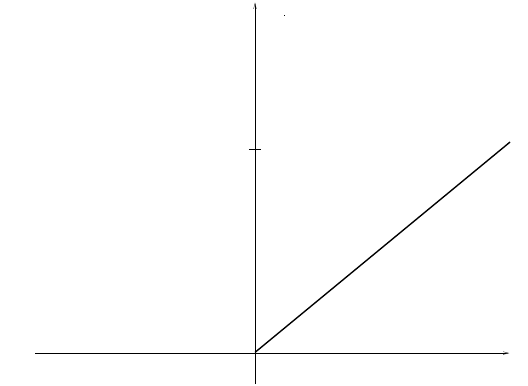
    \caption{The situation of Lemma~\ref{Lemma3} in the time interval $[0,\tau_{1}^\ell[$ for the first incoming road and the
      outgoing road in the coordinates, from left to right, $(\rho,\eta)$ and
      $(x, t)$.}\label{fig:tau1_Lemma3}
  \end{figure}  
  \begin{figure}[t!]
    \centering 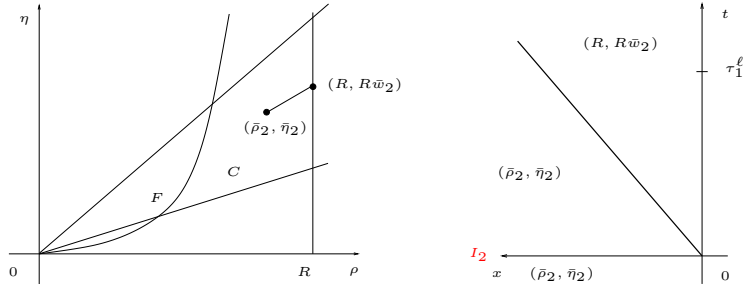 \hfil
    \input{rhoeta_tau1111.pdftex_t}
    \caption{The situation of Lemma~\ref{Lemma3} in the time interval $[0,\tau_{1}^\ell[$ for the second incoming road in
      the coordinates, from left to right, $(\rho,\eta)$ and
      $(x, t)$.}\label{fig:tau11_Lemma3}
  \end{figure}

  In the time interval $[\tau_{1}^\ell,\tau_{1}^\ell+\tau_{2}^\ell[$, the traffic lights becomes red for the road $I_1$
  and green for $I_2$ and the solution of the Riemann problem in the unique road given
  by the union of the $I_2$ and $I_3$ between
  $\left(R,R \bar w_2\right)$ and $(\bar \rho_{1}, \bar \eta_{1})$ is given by a rarefaction curve of the first family
  between $(R, R \bar w_2)$ and $\left(\rho_2^\flat, \eta_2^\flat\right)$
  followed by a linear wave between $\left(\rho_2^\flat, \eta_2^\flat\right)$ and $(\bar \rho_{1}, \bar \eta_{1})$, see Figure~\ref{fig:tau2_Lemma3}. 
  In $I_1$, instead, the flow at the junction is equal to zero and
  so the trace of the solution is $\left(R,R \bar w_1\right)$,
  see Figure~\ref{fig:tau22_Lemma3}. More precisely a shock wave with negative speed
  starts from the point $\left(\tau_1^\ell, 0\right)$ connecting the states
  $(\bar \rho_{1}, \bar \eta_{1})$ and $\left(R,R \bar w_1\right)$.
  \begin{figure}[t!]
    \centering 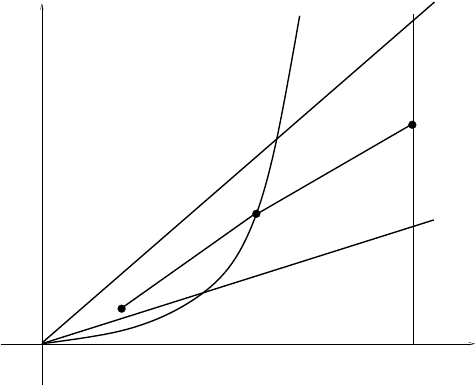 \hfil
    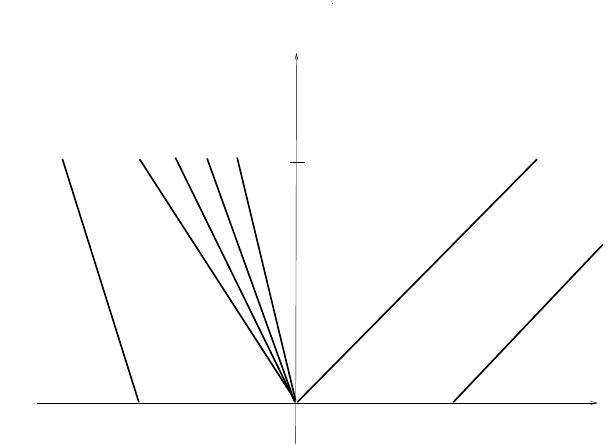
    \caption{The situation of Lemma~\ref{Lemma3} in the time interval $[\tau_{1}^\ell,\tau_{1}^\ell+\tau_{2}^\ell[$ for the second
      incoming road and the outgoing road in the coordinates, from
      left to right, $(\rho,\eta)$ and $(x, t)$.}\label{fig:tau2_Lemma3}
  \end{figure}
  \begin{figure}[t!]
    \centering 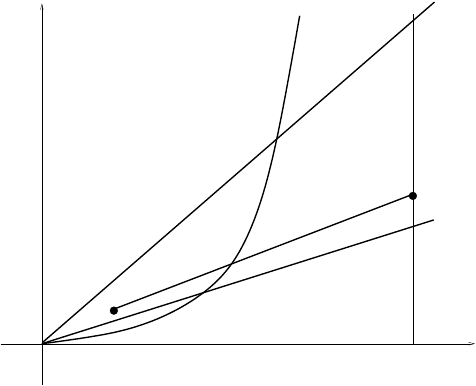 \hfil
    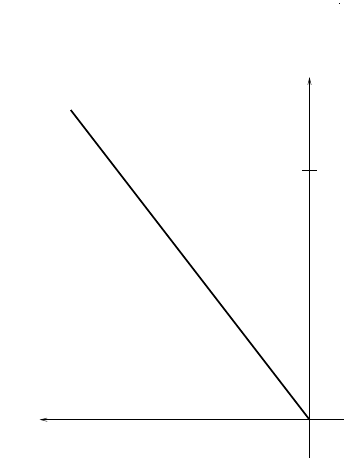
    \caption{The situation of Lemma~\ref{Lemma3} in the time interval $[\tau_{1}^\ell,\tau_{1}^\ell+\tau_{2}^\ell[$ for the first
      incoming road in the coordinates, from left to right,
      $(\rho,\eta)$ and $(x, t)$.}\label{fig:tau22_Lemma3}
  \end{figure}

  Similarly, in the time interval
  $[\tau_{1}^\ell + \tau_{2}^\ell, 2\tau_{1}^\ell + \tau_{2}^\ell[$,
  the traffic light is green for road $I_1$ and red for $I_2$;
  so we need to consider a Riemann
  problem between
  $(R,R \bar w_1)$ and $\left(\rho_2^\flat, \eta_2^\flat\right)$,
  see Figure~\ref{fig:tau3_Lemma3}. 
  The solution consists in a rarefaction curve of the first family
  between $(R,R \bar w_1)$ and $\left(\rho_1^\flat, \eta_1^\flat\right)$
  followed by a linear wave between
  $\left(\rho_1^\flat, \eta_1^\flat\right)$ and $\left(\rho_2^\flat, \eta_2^\flat\right)$.
  The situation of $I_2$ is analogous to that represented
  in Figure~\ref{fig:tau11_Lemma3}. More precisely at the point
  $\left(\tau_{1}^\ell + \tau_{2}^\ell, 0\right)$ a shock wave with negative
  speed is generated and it connects $\left(\rho_2^\flat, \eta_2^\flat\right)$
  with $\left(R,R \bar w_2\right)$.
  
  \begin{figure}[t!]
    \centering 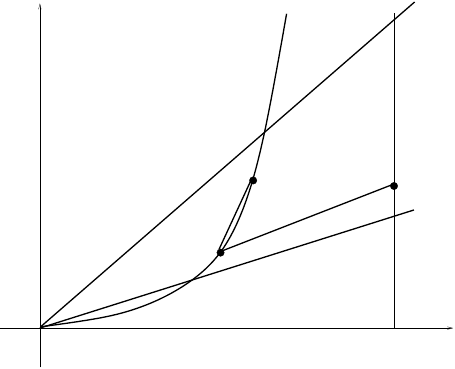 \hfil
    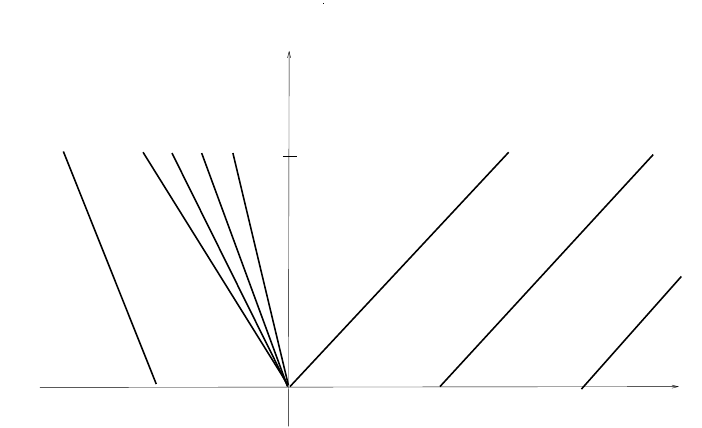
    \caption{The situation of Lemma~\ref{Lemma3} in the time interval $[\tau_{1}^\ell + \tau_{2}^\ell, 2\tau_{1}^\ell + \tau_{2}^\ell[$ for the
      first incoming road and the outgoing road in the coordinates,
      from left to right, $(\rho,\eta)$ and $(x, t)$.}\label{fig:tau3_Lemma3}
  \end{figure}
  \begin{figure}[t!]
    \centering 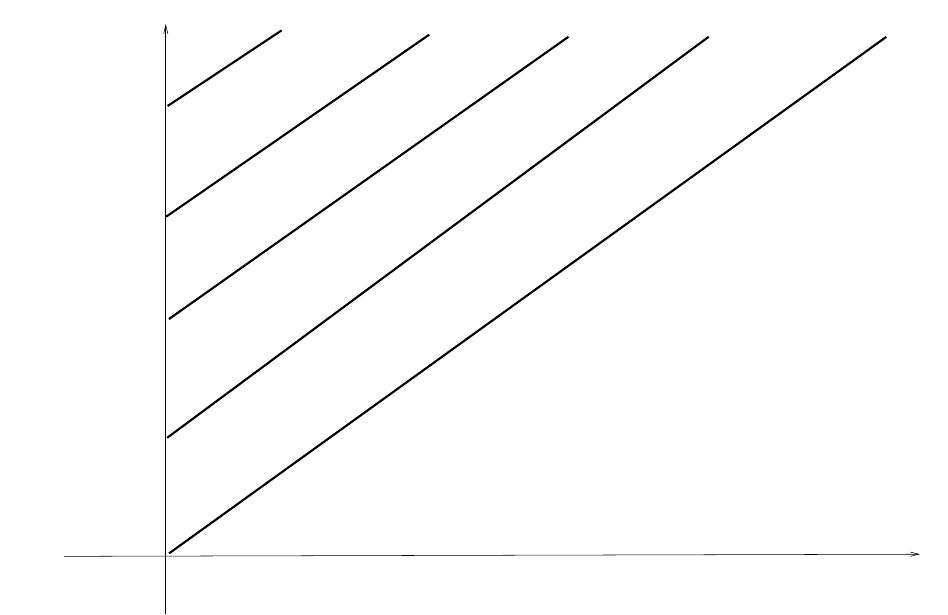
    \caption{The situation of Lemma~\ref{Lemma3} in the time interval $[0,2\tau_{1}^\ell+2\tau_{2}^\ell[$ for the outgoing road
      in the coordinates $(x, t)$. The two states $\left(\rho_1^\flat, \eta_1^\flat\right)$ and $\left(\rho_2^\flat, \eta_2^\flat\right)$, separated by linear waves, alternate periodically.}\label{fig:tau*_Lemma3}
  \end{figure}

 We proceed in the same way until we arrive at time $t = T$.
  The solution in $I_3$ is given by~(\ref{eq:I_3_Lemma3});
  see Figure~\ref{fig:tau*_Lemma3}. 
\end{proof}

\begin{lemma}
\label{Lemma4}
  Assume that the initial conditions satisfy
  \begin{equation*}
    \left(\bar \rho_{1}, \bar \eta_1\right) \in C, \quad
    \left(\bar \rho_{2}, \bar \eta_2\right) \in C, \quad
    \left(\bar \rho_{3}, \bar \eta_3\right) \in F.
  \end{equation*}
  Then the solution in the outgoing road $I_3$ is
  \begin{equation}
    \label{eq:I_3_Lemma4}
    \left(\rho_{\ell, 3}, \eta_{\ell, 3}\right) (t,x) =
    \left\{
      \begin{array}{l@{\quad}l}
      \left(\bar \rho_3, \bar \eta_3\right),
        &
          \textrm{ if }\, 0 < t < T, \quad x > V_{\max} t,
          \\
        \left(\rho_{1}^\flat, \eta_1^\flat\right),
        &
          \textrm{ if }\, (t, x) \in A_1^\ell,
        \\
        \left(\rho_{2}^\flat, \eta_2^\flat\right),
        &
          \textrm{ if }\, (t, x) \in A_2^\ell,
      \end{array}
    \right.
  \end{equation}
  where
  \begin{equation}
    \label{eq:A_1}
    A_1^\ell = \bigcup_{i=0}^{\ell - 1} \left\{(t, x):\,
      \begin{array}{c}
        0 < t < T, \quad x > 0 
        \\
        t - \tau_1^\ell - i \frac{T}{\ell} 
        < \frac{x}{V_{max}}
        < t - i \frac{T}{\ell}
      \end{array}
    \right\}
  \end{equation}
  and
  \begin{equation}
    \label{eq:A_2}
    A_2^\ell = \bigcup_{i=1}^{\ell} \left\{(t, x):\,
      \begin{array}{c}
        0 < t < T, \quad x > 0 
        \\
        t - i \frac{T}{\ell} 
        < \frac{x}{V_{max}}
        < t - i \frac{T}{\ell} + \tau_1^\ell
      \end{array}
    \right\}.
  \end{equation}
\end{lemma}

\begin{proof}
  We proceed in the same way of the previous lemmas. In the time interval
  $[0,\tau_{1}^\ell[$ the traffic lights is green 
  in the first incoming road and we study the Riemann problem
  as a classical one considering a unique road given
  by the union of the $I_1$ and $I_3$:
  the Riemann problem between $(\bar \rho_{1}, \bar \eta_{1})$
  and $\left(\bar \rho_3, \bar \eta_3\right)$
  produces a rarefaction curve of the first family
  between $(\bar \rho_{1}, \bar \eta_{1})$ and $\left(\rho_1^\flat, \eta_1^\flat\right)$
  followed by a linear wave between
  $\left(\rho_1^\flat, \eta_1^\flat\right)$ and $\left(\bar \rho_3, \bar \eta_3\right)$, see Figure~\ref{fig:tau1_Lemma4}. 
  In $I_2$, the flow at the junction is equal to zero and the solution in the road $I_2$ is given by a shock wave of the first family connecting
  $\left(\bar \rho_2, \bar \eta_2\right)$ to $\left(R,R \bar w_2\right)$, see Figure~\ref{fig:tau11_Lemma4}.
  \begin{figure}[t!]
    \centering
    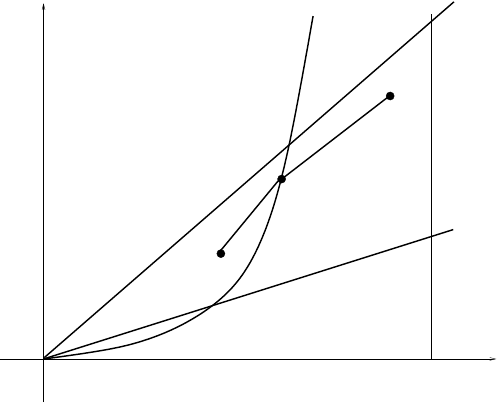
    \hfil
    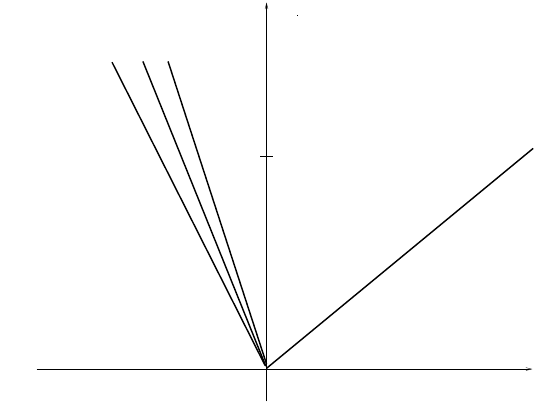
    \caption{The situation of Lemma~\ref{Lemma4} in the time interval $[0,\tau_{1}^\ell[$ for the first incoming road and the
      outgoing road in the coordinates, from left to right, $(\rho,\eta)$ and
      $(x, t)$.}\label{fig:tau1_Lemma4}
  \end{figure}  
  \begin{figure}[t!]
    \centering 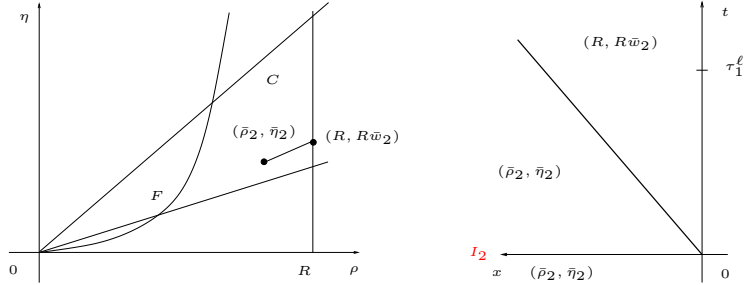 \hfil
    \input{rhoeta_tau1111.pdftex_t}
    \caption{The situation of Lemma~\ref{Lemma4} in the time interval $[0,\tau_{1}^\ell[$ for the second incoming road in
      the coordinates, from left to right, $(\rho,\eta)$ and
      $(x, t)$.}\label{fig:tau11_Lemma4}
  \end{figure}
  In the time interval $[\tau_{1}^\ell,\tau_{1}^\ell+\tau_{2}^\ell[$, when the traffic lights becomes red for the road $I_1$
  and green for $I_2$, the solution of the Riemann problem in the unique road given
  by the union of the $I_2$ and $I_3$ between
  $\left(R,R \bar w_2\right)$ and $\left(\rho_1^\flat, \eta_1^\flat\right)$ is given by a rarefaction curve of the first family
  between $(R, R \bar w_2)$ and $\left(\rho_2^\flat, \eta_2^\flat\right)$
  followed by a linear wave between $\left(\rho_2^\flat, \eta_2^\flat\right)$ and $\left(\rho_1^\flat, \eta_1^\flat\right)$, see Figure~\ref{fig:tau2_Lemma4}. 
  In $I_1$, instead, a shock wave with negative speed
  starts from the point $\left(\tau_1^\ell, 0\right)$ connecting the states
  $\left(\rho_1^\flat, \eta_1^\flat\right)$ and $\left(R,R \bar w_1\right)$, see Figure~\ref{fig:tau22_Lemma4}.
  \begin{figure}[t!]
    \centering 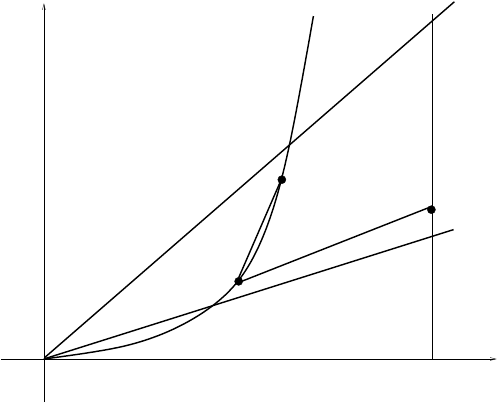 \hfil
    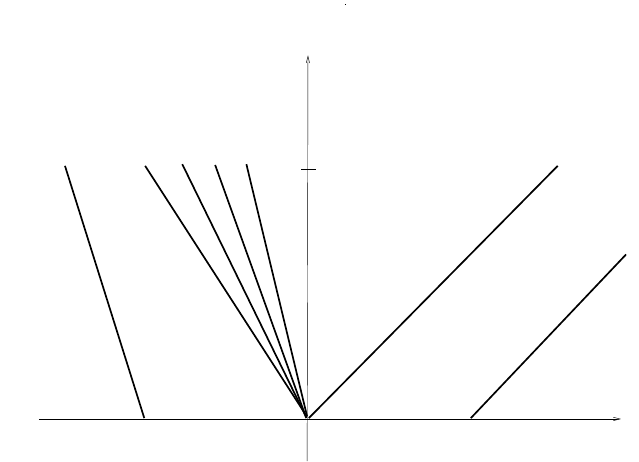
    \caption{The situation of Lemma~\ref{Lemma4} in the time interval $[\tau_{1}^\ell,\tau_{1}^\ell+\tau_{2}^\ell[$ for the second
      incoming road and the outgoing road in the coordinates, from
      left to right, $(\rho,\eta)$ and $(x, t)$.}\label{fig:tau2_Lemma4}
  \end{figure}
  \begin{figure}[t!]
    \centering 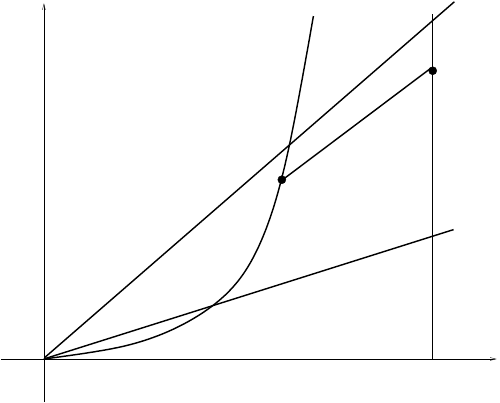 \hfil
    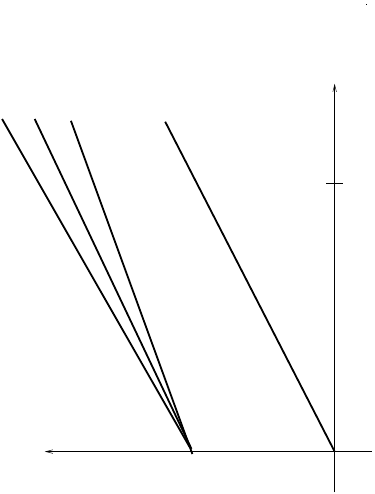
    \caption{The situation of Lemma~\ref{Lemma4} in the time interval $[\tau_{1}^\ell,\tau_{1}^\ell+\tau_{2}^\ell[$ for the first
      incoming road in the coordinates, from left to right,
      $(\rho,\eta)$ and $(x, t)$.}\label{fig:tau22_Lemma4}
  \end{figure}
  
From time $\tau_{1}^\ell+\tau_{2}^\ell$ to time $t = T$, the solution becomes as that described in the previous Lemma~\ref{Lemma3} and it is given by~(\ref{eq:I_3_Lemma4}) with a structure similar to Figure~\ref{fig:tau*_Lemma3}.  
 
 \end{proof} 

\subsection{The General Case}
\label{subsec:II}

Assume that the values $\sigma_1 > 0, \cdots, \sigma_{n-1} > 0, \sigma_n = 1$,
defined in~(\ref{eq:sigma}), are all positive constant, not depending
on the number of cycles $\ell$.

We can now state and prove the main result of the paper.
\begin{theorem}
  Assume~\textbf{\textup{(H-1)}}, \textbf{\textup{(H-2)}},
  and~\textbf{\textup{(H-3)}} hold.
  Fix, for every $i \in \left\{1, \ldots, n+1\right\}$, 
  $(\bar \rho_{i}, \bar \eta_{i}) \in F \cup C$.
  Consider the Riemann problem at the junction for the phase transition
  model~(\ref{eq:Modeleta}) where the initial conditions are given by
  $(\bar \rho_{1}, \bar \eta_{1}), \cdots, (\bar \rho_{n+1}, \bar \eta_{n+1})$. 

  Denote with $\left(\rho_{\ell,i} (t,x), \eta_{\ell,i} (t,x)\right)$
  ($i \in \left\{1, \cdots, n+1\right\}$)
  the solution to the Riemann problem, where the junction is governed by
  the traffic lights with $\ell$ cycles.

  If $\ell \to +\infty$, then there exists a function
  $\left(\widetilde \rho_{n+1} (t,x), \widetilde\eta_{n+1} (t,x)\right)$,
  defined in the outgoing road $I_{n+1}$, such that
  \begin{equation}
    \label{eq:convergence}
    \left(\rho_{\ell,n+1}, \eta_{\ell,n+1}\right) 
    \rightharpoonup^\ast 
    \left(\widetilde \rho_{n+1}, \widetilde\eta_{n+1}\right)
  \end{equation}
  converges 
  in the weak$^\ast$ topology of $\L{\infty}\left([0, T] \times I_{n+1}\right)$
  and the limit function
  $\left(\widetilde \rho_{n+1} (t,x), \widetilde\eta_{n+1} (t,x)\right)$
  is a weak solution to~(\ref{eq:Modeleta}) on $I_{n+1}$.

  Moreover, if $\left(\bar \rho_{n+1}, \bar \eta_{n+1}\right) \in F$,
  then
  \begin{equation}
    \label{eq:solution_F}
    \begin{split}
      \widetilde \rho_{n+1} (t,x) = & \left\{
        \begin{array}{ll}
          \bar \rho_{n+1}, 
          & \textrm{ if } x > V_{\max} t,\, 0 < t < T,
          \\
          \frac{\displaystyle 1}{\displaystyle
          \left[\sum_{i = 1}^{n} \sigma_i\right]}
          \displaystyle\sum_{i = 1}^n \sigma_i \rho_i^\flat,
          & \textrm{ if } 0 < x < V_{\max} t,\, 0 < t < T,
        \end{array}
      \right.
      \\
      \widetilde \eta_{n+1} (t,x) = & \left\{
        \begin{array}{ll}
          \bar \eta_{n+1}, 
          & \textrm{ if } x > V_{\max} t,\, 0 < t < T,
          \\
          \frac{\displaystyle 1}{\displaystyle
          \left[\sum_{i = 1}^{n} \sigma_i\right]}
          \displaystyle\sum_{i = 1}^n \sigma_i \eta_i^\flat,
          & \textrm{ if } 0 < x < V_{\max} t,\, 0 < t < T.
        \end{array}
      \right.
    \end{split}
  \end{equation}

  Instead, if $\left(\bar \rho_{n+1}, \bar \eta_{n+1}\right) \in C$,
  then
  \begin{equation}
    \label{eq:solution_C}
    \begin{split}
      \widetilde \rho_{n+1} (t,x) = & \left\{
        \begin{array}{ll}
          \bar \rho_{n+1}, 
          & \textrm{ if } x > \lambda t,\, 0 < t < T,
          \\
          \frac{\displaystyle 1}{\displaystyle
          \left[\sum_{i = 1}^{n} \sigma_i\right]}
          \displaystyle\sum_{i = 1}^n \sigma_i \rho_i^\sharp,
          & \textrm{ if } 0 < x < \lambda t,\, 0 < t < T,
        \end{array}
      \right.
      \\
      \widetilde \eta_{n+1} (t,x) = & \left\{
        \begin{array}{ll}
          \bar \eta_{n+1}, 
          & \textrm{ if } x > \lambda t,\, 0 < t < T,
          \\
          \frac{\displaystyle 1}{\displaystyle
          \left[\sum_{i = 1}^{n} \sigma_i\right]}
          \displaystyle\sum_{i = 1}^n \sigma_i \eta_i^\sharp,
          & \textrm{ if } 0 < x < \lambda t,\, 0 < t < T,
        \end{array}
      \right.
    \end{split}
  \end{equation}
  where $\lambda = v\left(\bar \rho_{n+1}, \bar \eta_{n+1}\right)$.

 Finally,
  for a.e. $t \in [0, T]$, the trace at $x=0^+$ of the maximal speed
  $\widetilde w_{n+1} = \frac{\widetilde \eta_{n+1}}{\widetilde \rho_{n+1}}$
  satisfies
  \begin{equation}
    \label{eq:cond_w_outgoing}
    \widetilde w_{n+1}(t, 0^+) 
    =
    \frac{1}{\sigma_1 + \cdots + \sigma_{n}}
    \left[\sigma_1 \bar w_1 + \cdots + \sigma_{n} \bar w_{n}
    \right].
  \end{equation}
\end{theorem}

\begin{proof}
  First consider the case $n=2$, i.e. the junction with
  $n = 2$ incoming roads and one outgoing road.
  In the time interval $[0, \tau_1^\ell[$ the traffic lights for the incoming
  road $I_2$ is red. Hence the trace of the solution in $I_2$ has maximal
  density $R$. This means that we may assume
  that
  \begin{equation}
    \label{eq:ic-I2}
    \left(\bar \rho_2, \bar \eta_2\right) = \left(R, \bar \eta_2\right)
    \in C.
  \end{equation}

  If the initial condition for the road $I_3$ belongs to the free phase $F$,
  then Lemma~\ref{Lemma3} and Lemma~\ref{Lemma4} imply that
  the sequence of solutions in $I_3$ is of Rademacker type,
  see for instance~\cite[Exercise~4.18]{MR2759829}.
  Hence we deduce that the limit
  of such sequence is given by~(\ref{eq:solution_F}).

  If the initial condition for the road $I_3$ belongs to the congested
  phase $C$,
  then Lemma~\ref{Lemma1} and Lemma~\ref{Lemma2} imply that
  the sequence of solutions in $I_3$ is again of Rademacker type,
  and so the limit
  of such sequence is given by~(\ref{eq:solution_C}).

  The functions~(\ref{eq:solution_F}) and~(\ref{eq:solution_C}) are piecewise
  constant and the discontinuity travels with speed satisfying
  the Rankine-Hugoniot condition. Hence they are weak solutions
  to~(\ref{eq:Modeleta}).

  Finally, consider the maximal speed
  \begin{equation*}
   \widetilde w_3(t, 0^+) = \frac{\widetilde \eta_3(t, 0^+)}
   {\widetilde \rho_3(t, 0^+)}
  \end{equation*}
  of the solution
  $\left(\widetilde \rho_3, \widetilde \eta_3\right)$ at the junction.
  If $\left(\bar \rho_3, \bar \eta_3\right) \in F$, then
  \begin{align*}
    \widetilde w_3(t, 0^+) 
    & = \frac{\sigma_1 \eta_1^\flat + \sigma_2 \eta_2^\flat}
      {\sigma_1 \rho_1^\flat + \sigma_2 \rho_2^\flat}
      = \frac{\sigma_1 \bar w_1\rho_1^\flat + \sigma_2 \bar w_2 \rho_2^\flat}
      {\sigma_1 \rho_1^\flat + \sigma_2 \rho_2^\flat}
\\
    & = 
      \frac{\sigma_1 \rho_1^\flat V_{\max}}{\sigma_1 \rho_1^\flat V_{\max}+\sigma_2 \rho_2^\flat V_{\max}} \bar w_{1} + \frac{\sigma_2\rho_2^\flat V_{\max}}{\sigma_1 \rho_1^\flat V_{\max}+\sigma_2 \rho_2^\flat V_{\max}} \bar w_{2}
  \\
    & = 
      \frac{\gamma_{1}}{\gamma_{1}+\gamma_{2}}\bar w_{1} + \frac{\gamma_{2}}{\gamma_{1}+\gamma_{2}} \bar w_{2}\,,
  \end{align*}
  with $\gamma_{1} = \sigma_1 \rho_1^\flat V_{\max}$ and
  $\gamma_{2} = \sigma_2 \rho_2^\flat V_{\max}$.
Therefore the maximal speed $\widetilde w(t,0^+)$ in the outgoing road is a convex combination of the maximal speeds in the two incoming roads $\bar w_{1}$ and $\bar w_{2}$ and it coincides with the condition on the maximal speed proposed in~\cite{GaravelloMarcellini}.

  If $\left(\bar \rho_3, \bar \eta_3\right) \in C$, then
  \begin{align*}
    \widetilde w_3(t, 0^+) 
    & = \frac{\sigma_1 \eta_1^\sharp + \sigma_2 \eta_2^\sharp}
      {\sigma_1 \rho_1^\sharp + \sigma_2 \rho_2^\sharp}
      = \frac{\sigma_1 \bar w_1\rho_1^\sharp + \sigma_2 \bar w_2 \rho_2^\sharp}
      {\sigma_1 \rho_1^\sharp + \sigma_2 \rho_2^\sharp}
\\
    & = \!
      \frac{\sigma_1 \rho_1^\sharp v\left(\bar \rho_3, \bar \eta_3\right)}
      {\sigma_1 \rho_1^\sharp v\left(\bar \rho_3, \bar \eta_3\right) +
      \sigma_2 \rho_2^\sharp v\left(\bar \rho_3, \bar \eta_3\right)} \bar w_{1}
      \!+ \!\frac{\sigma_2\rho_2^\sharp v\left(\bar \rho_3, \bar \eta_3\right)}
      {\sigma_1 \rho_1^\sharp v\left(\bar \rho_3, \bar \eta_3\right) 
      + \sigma_2 \rho_2^\sharp v\left(\bar \rho_3, \bar \eta_3\right)} \bar w_{2}
  \\
    & = 
      \frac{\gamma_{1}}{\gamma_{1}+\gamma_{2}}\bar w_{1} + \frac{\gamma_{2}}{\gamma_{1}+\gamma_{2}} \bar w_{2}\,,
  \end{align*}
  with $\gamma_{1} = \sigma_1 \rho_1^\sharp
  v\left(\rho_1^\sharp, \eta_1^\sharp\right)$ and
  $\gamma_{2} = \sigma_2 \rho_2^\sharp v\left(\rho_2^\sharp, \bar \eta_2^\sharp
  \right)$, since
  \begin{equation*}
    v\left(\rho_1^\sharp, \eta_1^\sharp\right) = 
    v\left(\rho_2^\sharp, \eta_2^\sharp\right) =
    v\left(\bar \rho_3, \bar \eta_3\right).
  \end{equation*}
  We conclude as in the previous case.

  The proof for the general case $n \ge 2$ can be obtained in a similar way.
  Indeed in the incoming roads where the traffic lights is red, the trace for
  the density is $R$. Instead, in the incoming road where the traffic lights
  is green, then the initial condition for such road is propagated in $I_{n+1}$
  with speed $v\left(\bar \rho_{n+1}, \bar \eta_{n+1}\right)$.
  Hence the solution
  $\left(\tilde \rho_{\ell, n+1}, \tilde \eta_{\ell, n+1}\right)$ is similar
  to that of Lemmas~\ref{Lemma1}-\ref{Lemma4} in the sense that
  there exist subsets $A_1^\ell, \cdots A_n^\ell$ of $[0,T] \times I_{n+1}$
  with a ``periodic'' structure in which 
  $\left(\tilde \rho_{\ell, n+1}, \tilde \eta_{\ell, n+1}\right)$
  is given respectively by
  $\left(\bar \rho_1, \bar \eta_1\right), \cdots,
  \left(\bar \rho_n, \bar \eta_n\right)$. This permits to conclude.
\end{proof}

\section*{Acknowledgments}
The authors were partially supported by the INdAM-GNAMPA 2017 project ``Conservation Laws: from Theory to Technology''.

{\small{

    \bibliographystyle{abbrv}

    \bibliography{model} }}
\end{document}

%% file: rhoeta.pdftex_t
\begin{picture}(0,0)%
\includegraphics{rhoeta.pdf}%
\end{picture}%
%
%
\setlength{\unitlength}{868sp}%
\begingroup\makeatletter\ifx\SetFigFont\undefined%
\gdef\SetFigFont#1#2#3#4#5{%
  \reset@font\fontsize{#1}{#2pt}%
  \fontfamily{#3}\fontseries{#4}\fontshape{#5}%
  \selectfont}%
\fi\endgroup%
\begin{picture}(10374,8370)(889,-8398)
\put(1201,-556){\makebox(0,0)[lb]{\smash{{\SetFigFont{5}{6.0}{\rmdefault}{\mddefault}{\updefault}{\color[rgb]{0,0,0}$\eta$}%
}}}}
\put(5446,-5071){\makebox(0,0)[lb]{\smash{{\SetFigFont{5}{6.0}{\rmdefault}{\mddefault}{\updefault}{\color[rgb]{0,0,0}$F$}%
}}}}
\put(961,-8191){\makebox(0,0)[lb]{\smash{{\SetFigFont{5}{6.0}{\rmdefault}{\mddefault}{\updefault}{\color[rgb]{0,0,0}$0$}%
}}}}
\put(10966,-8056){\makebox(0,0)[lb]{\smash{{\SetFigFont{5}{6.0}{\rmdefault}{\mddefault}{\updefault}{\color[rgb]{0,0,0}$\rho$}%
}}}}
\put(9436,-8146){\makebox(0,0)[lb]{\smash{{\SetFigFont{5}{6.0}{\rmdefault}{\mddefault}{\updefault}{\color[rgb]{0,0,0}$R$}%
}}}}
\put(8296,-3871){\makebox(0,0)[lb]{\smash{{\SetFigFont{5}{6.0}{\rmdefault}{\mddefault}{\updefault}{\color[rgb]{0,0,0}$C$}%
}}}}
\end{picture}%

%% file: rhov.pdftex_t
\begin{picture}(0,0)%
\includegraphics{rhov.pdf}%
\end{picture}%
%
%
\setlength{\unitlength}{868sp}%
\begingroup\makeatletter\ifx\SetFigFont\undefined%
\gdef\SetFigFont#1#2#3#4#5{%
  \reset@font\fontsize{#1}{#2pt}%
  \fontfamily{#3}\fontseries{#4}\fontshape{#5}%
  \selectfont}%
\fi\endgroup%
\begin{picture}(10377,8349)(886,-8398)
\put(11056,-8191){\makebox(0,0)[lb]{\smash{{\SetFigFont{5}{6.0}{\familydefault}{\mddefault}{\updefault}{\color[rgb]{0,0,0}$\rho$}%
}}}}
\put(2476,-4786){\makebox(0,0)[lb]{\smash{{\SetFigFont{5}{6.0}{\familydefault}{\mddefault}{\updefault}{\color[rgb]{0,0,0}$F$}%
}}}}
\put(6316,-2701){\makebox(0,0)[lb]{\smash{{\SetFigFont{5}{6.0}{\familydefault}{\mddefault}{\updefault}{\color[rgb]{0,0,0}$C$}%
}}}}
\put(1111,-8131){\makebox(0,0)[lb]{\smash{{\SetFigFont{5}{6.0}{\familydefault}{\mddefault}{\updefault}{\color[rgb]{0,0,0}$0$}%
}}}}
\put(9691,-8191){\makebox(0,0)[lb]{\smash{{\SetFigFont{5}{6.0}{\familydefault}{\mddefault}{\updefault}{\color[rgb]{0,0,0}$R$}%
}}}}
\put(901,-511){\makebox(0,0)[lb]{\smash{{\SetFigFont{5}{6.0}{\familydefault}{\mddefault}{\updefault}{\color[rgb]{0,0,0}$\rho v$}%
}}}}
\end{picture}%

%% file: rhoeta_1.pdftex_t
\begin{picture}(0,0)%
\includegraphics{rhoeta_1.pdf}%
\end{picture}%
%
%
\setlength{\unitlength}{947sp}%
\begingroup\makeatletter\ifx\SetFigFont\undefined%
\gdef\SetFigFont#1#2#3#4#5{%
  \reset@font\fontsize{#1}{#2pt}%
  \fontfamily{#3}\fontseries{#4}\fontshape{#5}%
  \selectfont}%
\fi\endgroup%
\begin{picture}(10374,8370)(889,-8398)
\put(1201,-556){\makebox(0,0)[lb]{\smash{{\SetFigFont{5}{6.0}{\rmdefault}{\mddefault}{\updefault}{\color[rgb]{0,0,0}$\eta$}%
}}}}
\put(5446,-5311){\makebox(0,0)[lb]{\smash{{\SetFigFont{5}{6.0}{\rmdefault}{\mddefault}{\updefault}{\color[rgb]{0,0,0}$F$}%
}}}}
\put(961,-8191){\makebox(0,0)[lb]{\smash{{\SetFigFont{5}{6.0}{\rmdefault}{\mddefault}{\updefault}{\color[rgb]{0,0,0}$0$}%
}}}}
\put(10966,-8056){\makebox(0,0)[lb]{\smash{{\SetFigFont{5}{6.0}{\rmdefault}{\mddefault}{\updefault}{\color[rgb]{0,0,0}$\rho$}%
}}}}
\put(9436,-8146){\makebox(0,0)[lb]{\smash{{\SetFigFont{5}{6.0}{\rmdefault}{\mddefault}{\updefault}{\color[rgb]{0,0,0}$R$}%
}}}}
\put(7261,-4291){\makebox(0,0)[lb]{\smash{{\SetFigFont{5}{6.0}{\rmdefault}{\mddefault}{\updefault}{\color[rgb]{0,0,0}$C$}%
}}}}
\put(3226,-4516){\makebox(0,0)[lb]{\smash{{\SetFigFont{5}{6.0}{\rmdefault}{\mddefault}{\updefault}{\color[rgb]{0,0,0}$(\bar \rho_{1}, \bar \eta_{1})$}%
}}}}
\put(3256,-8191){\makebox(0,0)[lb]{\smash{{\SetFigFont{5}{6.0}{\rmdefault}{\mddefault}{\updefault}{\color[rgb]{0,0,0}$(\bar \rho_{2}, \bar \eta_{2})$}%
}}}}
\put(10141,-2926){\makebox(0,0)[lb]{\smash{{\SetFigFont{5}{6.0}{\rmdefault}{\mddefault}{\updefault}{\color[rgb]{0,0,0}$(\bar \rho_{3}, \bar \eta_{3})$}%
}}}}
\put(7846,-6061){\makebox(0,0)[lb]{\smash{{\SetFigFont{5}{6.0}{\rmdefault}{\mddefault}{\updefault}{\color[rgb]{0,0,0}$\left(\rho_2^\sharp, \eta_2^\sharp\right)$}%
}}}}
\put(7846,-496){\makebox(0,0)[lb]{\smash{{\SetFigFont{5}{6.0}{\rmdefault}{\mddefault}{\updefault}{\color[rgb]{0,0,0}$\left(\rho_1^\sharp, \eta_1^\sharp\right)$}%
}}}}
\end{picture}%

%% file: rhoeta_2.pdftex_t
\begin{picture}(0,0)%
\includegraphics{rhoeta_2.pdf}%
\end{picture}%
%
%
\setlength{\unitlength}{947sp}%
\begingroup\makeatletter\ifx\SetFigFont\undefined%
\gdef\SetFigFont#1#2#3#4#5{%
  \reset@font\fontsize{#1}{#2pt}%
  \fontfamily{#3}\fontseries{#4}\fontshape{#5}%
  \selectfont}%
\fi\endgroup%
\begin{picture}(10374,8370)(889,-8398)
\put(1201,-556){\makebox(0,0)[lb]{\smash{{\SetFigFont{5}{6.0}{\rmdefault}{\mddefault}{\updefault}{\color[rgb]{0,0,0}$\eta$}%
}}}}
\put(5131,-5506){\makebox(0,0)[lb]{\smash{{\SetFigFont{5}{6.0}{\rmdefault}{\mddefault}{\updefault}{\color[rgb]{0,0,0}$F$}%
}}}}
\put(961,-8191){\makebox(0,0)[lb]{\smash{{\SetFigFont{5}{6.0}{\rmdefault}{\mddefault}{\updefault}{\color[rgb]{0,0,0}$0$}%
}}}}
\put(10966,-8056){\makebox(0,0)[lb]{\smash{{\SetFigFont{5}{6.0}{\rmdefault}{\mddefault}{\updefault}{\color[rgb]{0,0,0}$\rho$}%
}}}}
\put(9436,-8146){\makebox(0,0)[lb]{\smash{{\SetFigFont{5}{6.0}{\rmdefault}{\mddefault}{\updefault}{\color[rgb]{0,0,0}$R$}%
}}}}
\put(9061,-3301){\makebox(0,0)[lb]{\smash{{\SetFigFont{5}{6.0}{\rmdefault}{\mddefault}{\updefault}{\color[rgb]{0,0,0}$C$}%
}}}}
\put(2716,-5116){\makebox(0,0)[lb]{\smash{{\SetFigFont{5}{6.0}{\rmdefault}{\mddefault}{\updefault}{\color[rgb]{0,0,0}$(\bar \rho_{1}, \bar \eta_{1})$}%
}}}}
\put(2941,-8041){\makebox(0,0)[lb]{\smash{{\SetFigFont{5}{6.0}{\rmdefault}{\mddefault}{\updefault}{\color[rgb]{0,0,0}$(\bar \rho_{2}, \bar \eta_{2})$}%
}}}}
\put(5731,-6721){\makebox(0,0)[lb]{\smash{{\SetFigFont{5}{6.0}{\rmdefault}{\mddefault}{\updefault}{\color[rgb]{0,0,0}$\left(\rho_2^\flat, \eta_2^\flat\right)$}%
}}}}
\put(5131,-2386){\makebox(0,0)[lb]{\smash{{\SetFigFont{5}{6.0}{\rmdefault}{\mddefault}{\updefault}{\color[rgb]{0,0,0}$\left(\rho_1^\flat, \eta_1^\flat\right)$}%
}}}}
\end{picture}%

%% file: rhoeta_tau1.pdftex_t
\begin{picture}(0,0)%
\includegraphics{rhoeta_tau1.pdf}%
\end{picture}%
%
%
\setlength{\unitlength}{908sp}%
\begingroup\makeatletter\ifx\SetFigFont\undefined%
\gdef\SetFigFont#1#2#3#4#5{%
  \reset@font\fontsize{#1}{#2pt}%
  \fontfamily{#3}\fontseries{#4}\fontshape{#5}%
  \selectfont}%
\fi\endgroup%
\begin{picture}(10374,8399)(889,-8398)
\put(1201,-556){\makebox(0,0)[lb]{\smash{{\SetFigFont{5}{6.0}{\rmdefault}{\mddefault}{\updefault}{\color[rgb]{0,0,0}$\eta$}%
}}}}
\put(5716,-4696){\makebox(0,0)[lb]{\smash{{\SetFigFont{5}{6.0}{\rmdefault}{\mddefault}{\updefault}{\color[rgb]{0,0,0}$F$}%
}}}}
\put(10966,-8056){\makebox(0,0)[lb]{\smash{{\SetFigFont{5}{6.0}{\rmdefault}{\mddefault}{\updefault}{\color[rgb]{0,0,0}$\rho$}%
}}}}
\put(9436,-8146){\makebox(0,0)[lb]{\smash{{\SetFigFont{5}{6.0}{\rmdefault}{\mddefault}{\updefault}{\color[rgb]{0,0,0}$R$}%
}}}}
\put(8236,-3271){\makebox(0,0)[lb]{\smash{{\SetFigFont{5}{6.0}{\rmdefault}{\mddefault}{\updefault}{\color[rgb]{0,0,0}$C$}%
}}}}
\put(6961,-3931){\makebox(0,0)[lb]{\smash{{\SetFigFont{5}{6.0}{\rmdefault}{\mddefault}{\updefault}{\color[rgb]{0,0,0}$(\bar \rho_{1}, \bar \eta_{1})$}%
}}}}
\put(931,-8056){\makebox(0,0)[lb]{\smash{{\SetFigFont{5}{6.0}{\rmdefault}{\mddefault}{\updefault}{\color[rgb]{0,0,0}$0$}%
}}}}
\put(7951,-5851){\makebox(0,0)[lb]{\smash{{\SetFigFont{5}{6.0}{\rmdefault}{\mddefault}{\updefault}{\color[rgb]{0,0,0}$\left(\bar \rho_3, \bar \eta_3\right)$}%
}}}}
\put(10066,-1531){\makebox(0,0)[lb]{\smash{{\SetFigFont{5}{6.0}{\rmdefault}{\mddefault}{\updefault}{\color[rgb]{0,0,0}$\left(\rho_1^\sharp, \eta_1^\sharp\right)$}%
}}}}
\end{picture}%

%% file: rhoeta_tau11.pdftex_t
\begin{picture}(0,0)%
\includegraphics{rhoeta_tau11.pdf}%
\end{picture}%
%
%
\setlength{\unitlength}{908sp}%
\begingroup\makeatletter\ifx\SetFigFont\undefined%
\gdef\SetFigFont#1#2#3#4#5{%
  \reset@font\fontsize{#1}{#2pt}%
  \fontfamily{#3}\fontseries{#4}\fontshape{#5}%
  \selectfont}%
\fi\endgroup%
\begin{picture}(11445,8366)(136,-8190)
\put(1936,-8011){\makebox(0,0)[lb]{\smash{{\SetFigFont{5}{6.0}{\rmdefault}{\mddefault}{\updefault}{\color[rgb]{0,0,0}$(\bar \rho_{1}, \bar \eta_{1})$}%
}}}}
\put(4966,-76){\makebox(0,0)[lb]{\smash{{\SetFigFont{5}{6.0}{\rmdefault}{\mddefault}{\updefault}{\color[rgb]{0,0,0}$t$}%
}}}}
\put(10786,-8116){\makebox(0,0)[lb]{\smash{{\SetFigFont{5}{6.0}{\rmdefault}{\mddefault}{\updefault}{\color[rgb]{0,0,0}$x$}%
}}}}
\put(1486,-4291){\makebox(0,0)[lb]{\smash{{\SetFigFont{5}{6.0}{\rmdefault}{\mddefault}{\updefault}{\color[rgb]{0,0,0}$(\bar \rho_{1}, \bar \eta_{1})$}%
}}}}
\put(5101,-8026){\makebox(0,0)[lb]{\smash{{\SetFigFont{5}{6.0}{\rmdefault}{\mddefault}{\updefault}{\color[rgb]{0,0,0}$0$}%
}}}}
\put(7981,-5536){\makebox(0,0)[lb]{\smash{{\SetFigFont{5}{6.0}{\rmdefault}{\mddefault}{\updefault}{\color[rgb]{0,0,0}$\left(\bar \rho_3, \bar \eta_3\right)$}%
}}}}
\put(7471,-8041){\makebox(0,0)[lb]{\smash{{\SetFigFont{5}{6.0}{\rmdefault}{\mddefault}{\updefault}{\color[rgb]{0,0,0}$\left(\bar \rho_3, \bar \eta_3\right)$}%
}}}}
\put(151,-7501){\makebox(0,0)[lb]{\smash{{\SetFigFont{5}{6.0}{\rmdefault}{\mddefault}{\updefault}{\color[rgb]{1,0,0}$I_1$}%
}}}}
\put(11566,-7576){\makebox(0,0)[lb]{\smash{{\SetFigFont{5}{6.0}{\rmdefault}{\mddefault}{\updefault}{\color[rgb]{1,0,0}$I_3$}%
}}}}
\put(5911,-1981){\makebox(0,0)[lb]{\smash{{\SetFigFont{5}{6.0}{\rmdefault}{\mddefault}{\updefault}{\color[rgb]{0,0,0}$\left(\rho_1^\sharp, \eta_1^\sharp\right)$}%
}}}}
\put(4591,-3061){\makebox(0,0)[lb]{\smash{{\SetFigFont{5}{6.0}{\rmdefault}{\mddefault}{\updefault}{\color[rgb]{0,0,0}$\tau_{1}^\ell$}%
}}}}
\end{picture}%

%% file: rhoeta_tau111.pdftex_t
\begin{picture}(0,0)%
\includegraphics{rhoeta_tau111.pdf}%
\end{picture}%
%
%
\setlength{\unitlength}{868sp}%
\begingroup\makeatletter\ifx\SetFigFont\undefined%
\gdef\SetFigFont#1#2#3#4#5{%
  \reset@font\fontsize{#1}{#2pt}%
  \fontfamily{#3}\fontseries{#4}\fontshape{#5}%
  \selectfont}%
\fi\endgroup%
\begin{picture}(10377,8399)(886,-8398)
\put(1201,-556){\makebox(0,0)[lb]{\smash{{\SetFigFont{5}{6.0}{\rmdefault}{\mddefault}{\updefault}{\color[rgb]{0,0,0}$\eta$}%
}}}}
\put(5071,-5926){\makebox(0,0)[lb]{\smash{{\SetFigFont{5}{6.0}{\rmdefault}{\mddefault}{\updefault}{\color[rgb]{0,0,0}$F$}%
}}}}
\put(10966,-8056){\makebox(0,0)[lb]{\smash{{\SetFigFont{5}{6.0}{\rmdefault}{\mddefault}{\updefault}{\color[rgb]{0,0,0}$\rho$}%
}}}}
\put(9436,-8146){\makebox(0,0)[lb]{\smash{{\SetFigFont{5}{6.0}{\rmdefault}{\mddefault}{\updefault}{\color[rgb]{0,0,0}$R$}%
}}}}
\put(8296,-3871){\makebox(0,0)[lb]{\smash{{\SetFigFont{5}{6.0}{\rmdefault}{\mddefault}{\updefault}{\color[rgb]{0,0,0}$C$}%
}}}}
\put(901,-8131){\makebox(0,0)[lb]{\smash{{\SetFigFont{5}{6.0}{\rmdefault}{\mddefault}{\updefault}{\color[rgb]{0,0,0}$0$}%
}}}}
\put(6781,-4546){\makebox(0,0)[lb]{\smash{{\SetFigFont{5}{6.0}{\rmdefault}{\mddefault}{\updefault}{\color[rgb]{0,0,0}$(\bar \rho_{2}, \bar \eta_{2})$}%
}}}}
\put(10321,-2536){\makebox(0,0)[lb]{\smash{{\SetFigFont{5}{6.0}{\rmdefault}{\mddefault}{\updefault}{\color[rgb]{0,0,0}$\left(R,R \bar w_2\right)$}%
}}}}
\end{picture}%

%% file: rhoeta_tau1111.pdftex_t
\begin{picture}(0,0)%
\includegraphics{rhoeta_tau1111.pdf}%
\end{picture}%
%
%
\setlength{\unitlength}{1026sp}%
\begingroup\makeatletter\ifx\SetFigFont\undefined%
\gdef\SetFigFont#1#2#3#4#5{%
  \reset@font\fontsize{#1}{#2pt}%
  \fontfamily{#3}\fontseries{#4}\fontshape{#5}%
  \selectfont}%
\fi\endgroup%
\begin{picture}(6207,8109)(121,-8173)
\put(6166,-1741){\makebox(0,0)[lb]{\smash{{\SetFigFont{5}{6.0}{\rmdefault}{\mddefault}{\updefault}{\color[rgb]{0,0,0}$t$}%
}}}}
\put(676,-7996){\makebox(0,0)[lb]{\smash{{\SetFigFont{5}{6.0}{\rmdefault}{\mddefault}{\updefault}{\color[rgb]{0,0,0}$x$}%
}}}}
\put(6151,-8056){\makebox(0,0)[lb]{\smash{{\SetFigFont{5}{6.0}{\rmdefault}{\mddefault}{\updefault}{\color[rgb]{0,0,0}$0$}%
}}}}
\put(1576,-7996){\makebox(0,0)[lb]{\smash{{\SetFigFont{5}{6.0}{\rmdefault}{\mddefault}{\updefault}{\color[rgb]{0,0,0}$(\bar \rho_{2}, \bar \eta_{2})$}%
}}}}
\put(811,-5566){\makebox(0,0)[lb]{\smash{{\SetFigFont{5}{6.0}{\rmdefault}{\mddefault}{\updefault}{\color[rgb]{0,0,0}$(\bar \rho_{2}, \bar \eta_{2})$}%
}}}}
\put(2851,-2461){\makebox(0,0)[lb]{\smash{{\SetFigFont{5}{6.0}{\rmdefault}{\mddefault}{\updefault}{\color[rgb]{0,0,0}$\left(R,R \bar w_2\right)$}%
}}}}
\put(136,-7516){\makebox(0,0)[lb]{\smash{{\SetFigFont{5}{6.0}{\rmdefault}{\mddefault}{\updefault}{\color[rgb]{1,0,0}$I_2$}%
}}}}
\put(6256,-3046){\makebox(0,0)[lb]{\smash{{\SetFigFont{5}{6.0}{\rmdefault}{\mddefault}{\updefault}{\color[rgb]{0,0,0}$\tau_{1}^\ell$}%
}}}}
\end{picture}%

%% file: rhoeta_tau2.pdftex_t
\begin{picture}(0,0)%
\includegraphics{rhoeta_tau2.pdf}%
\end{picture}%
%
%
\setlength{\unitlength}{829sp}%
\begingroup\makeatletter\ifx\SetFigFont\undefined%
\gdef\SetFigFont#1#2#3#4#5{%
  \reset@font\fontsize{#1}{#2pt}%
  \fontfamily{#3}\fontseries{#4}\fontshape{#5}%
  \selectfont}%
\fi\endgroup%
\begin{picture}(10377,8399)(886,-8398)
\put(1201,-556){\makebox(0,0)[lb]{\smash{{\SetFigFont{5}{6.0}{\rmdefault}{\mddefault}{\updefault}{\color[rgb]{0,0,0}$\eta$}%
}}}}
\put(4141,-6106){\makebox(0,0)[lb]{\smash{{\SetFigFont{5}{6.0}{\rmdefault}{\mddefault}{\updefault}{\color[rgb]{0,0,0}$F$}%
}}}}
\put(10966,-8056){\makebox(0,0)[lb]{\smash{{\SetFigFont{5}{6.0}{\rmdefault}{\mddefault}{\updefault}{\color[rgb]{0,0,0}$\rho$}%
}}}}
\put(9436,-8146){\makebox(0,0)[lb]{\smash{{\SetFigFont{5}{6.0}{\rmdefault}{\mddefault}{\updefault}{\color[rgb]{0,0,0}$R$}%
}}}}
\put(7861,-3076){\makebox(0,0)[lb]{\smash{{\SetFigFont{5}{6.0}{\rmdefault}{\mddefault}{\updefault}{\color[rgb]{0,0,0}$C$}%
}}}}
\put(901,-8131){\makebox(0,0)[lb]{\smash{{\SetFigFont{5}{6.0}{\rmdefault}{\mddefault}{\updefault}{\color[rgb]{0,0,0}$0$}%
}}}}
\put(10231,-2806){\makebox(0,0)[lb]{\smash{{\SetFigFont{5}{6.0}{\rmdefault}{\mddefault}{\updefault}{\color[rgb]{0,0,0}$(R, R \bar w_2)$}%
}}}}
\put(7741,-3886){\makebox(0,0)[lb]{\smash{{\SetFigFont{5}{6.0}{\rmdefault}{\mddefault}{\updefault}{\color[rgb]{0,0,0}$\left(\rho_2^\sharp, \eta_2^\sharp\right)$}%
}}}}
\put(10186,-1426){\makebox(0,0)[lb]{\smash{{\SetFigFont{5}{6.0}{\rmdefault}{\mddefault}{\updefault}{\color[rgb]{0,0,0}$\left(\rho_1^\sharp, \eta_1^\sharp\right)$}%
}}}}
\end{picture}%

%% file: rhoeta_tau22.pdftex_t
\begin{picture}(0,0)%
\includegraphics{rhoeta_tau22.pdf}%
\end{picture}%
%
%
\setlength{\unitlength}{1026sp}%
\begingroup\makeatletter\ifx\SetFigFont\undefined%
\gdef\SetFigFont#1#2#3#4#5{%
  \reset@font\fontsize{#1}{#2pt}%
  \fontfamily{#3}\fontseries{#4}\fontshape{#5}%
  \selectfont}%
\fi\endgroup%
\begin{picture}(11913,8154)(211,-8248)
\put(6661,-7981){\makebox(0,0)[lb]{\smash{{\SetFigFont{5}{6.0}{\rmdefault}{\mddefault}{\updefault}{\color[rgb]{0,0,0}$\left(\rho_1^\sharp, \eta_1^\sharp\right)$}%
}}}}
\put(4966,-1261){\makebox(0,0)[lb]{\smash{{\SetFigFont{5}{6.0}{\rmdefault}{\mddefault}{\updefault}{\color[rgb]{0,0,0}$t$}%
}}}}
\put(10921,-8041){\makebox(0,0)[lb]{\smash{{\SetFigFont{5}{6.0}{\rmdefault}{\mddefault}{\updefault}{\color[rgb]{0,0,0}$x$}%
}}}}
\put(5821,-7996){\makebox(0,0)[lb]{\smash{{\SetFigFont{5}{6.0}{\rmdefault}{\mddefault}{\updefault}{\color[rgb]{0,0,0}$\tau_1^\ell$}%
}}}}
\put(1066,-7921){\makebox(0,0)[lb]{\smash{{\SetFigFont{5}{6.0}{\rmdefault}{\mddefault}{\updefault}{\color[rgb]{0,0,0}$(\bar \rho_{2}, \bar \eta_{2})$}%
}}}}
\put(6106,-3151){\makebox(0,0)[lb]{\smash{{\SetFigFont{5}{6.0}{\rmdefault}{\mddefault}{\updefault}{\color[rgb]{0,0,0}$\tau_{1}^\ell+\tau_{2}^\ell$}%
}}}}
\put(7321,-6136){\makebox(0,0)[lb]{\smash{{\SetFigFont{5}{6.0}{\rmdefault}{\mddefault}{\updefault}{\color[rgb]{0,0,0}$\left(\rho_1^\sharp, \eta_1^\sharp\right)$}%
}}}}
\put(2581,-6166){\makebox(0,0)[lb]{\smash{{\SetFigFont{5}{6.0}{\rmdefault}{\mddefault}{\updefault}{\color[rgb]{0,0,0}$(R, R \bar w_2)$}%
}}}}
\put(3061,-7951){\makebox(0,0)[lb]{\smash{{\SetFigFont{5}{6.0}{\rmdefault}{\mddefault}{\updefault}{\color[rgb]{0,0,0}$(R, R \bar w_2)$}%
}}}}
\put(226,-7531){\makebox(0,0)[lb]{\smash{{\SetFigFont{5}{6.0}{\rmdefault}{\mddefault}{\updefault}{\color[rgb]{1,0,0}$I_2$}%
}}}}
\put(11536,-7561){\makebox(0,0)[lb]{\smash{{\SetFigFont{5}{6.0}{\rmdefault}{\mddefault}{\updefault}{\color[rgb]{1,0,0}$I_3$}%
}}}}
\put(8641,-8011){\makebox(0,0)[lb]{\smash{{\SetFigFont{5}{6.0}{\rmdefault}{\mddefault}{\updefault}{\color[rgb]{0,0,0}$(\bar \rho_{3}, \bar \eta_{3})$}%
}}}}
\put(6001,-3826){\makebox(0,0)[lb]{\smash{{\SetFigFont{5}{6.0}{\rmdefault}{\mddefault}{\updefault}{\color[rgb]{0,0,0}$\left(\rho_2^\sharp, \eta_2^\sharp\right)$}%
}}}}
\end{picture}%

%% file: rhoeta_tau222.pdftex_t
\begin{picture}(0,0)%
\includegraphics{rhoeta_tau222.pdf}%
\end{picture}%
%
%
\setlength{\unitlength}{868sp}%
\begingroup\makeatletter\ifx\SetFigFont\undefined%
\gdef\SetFigFont#1#2#3#4#5{%
  \reset@font\fontsize{#1}{#2pt}%
  \fontfamily{#3}\fontseries{#4}\fontshape{#5}%
  \selectfont}%
\fi\endgroup%
\begin{picture}(10377,8399)(886,-8398)
\put(1201,-556){\makebox(0,0)[lb]{\smash{{\SetFigFont{5}{6.0}{\rmdefault}{\mddefault}{\updefault}{\color[rgb]{0,0,0}$\eta$}%
}}}}
\put(5536,-5131){\makebox(0,0)[lb]{\smash{{\SetFigFont{5}{6.0}{\rmdefault}{\mddefault}{\updefault}{\color[rgb]{0,0,0}$F$}%
}}}}
\put(10966,-8056){\makebox(0,0)[lb]{\smash{{\SetFigFont{5}{6.0}{\rmdefault}{\mddefault}{\updefault}{\color[rgb]{0,0,0}$\rho$}%
}}}}
\put(9436,-8146){\makebox(0,0)[lb]{\smash{{\SetFigFont{5}{6.0}{\rmdefault}{\mddefault}{\updefault}{\color[rgb]{0,0,0}$R$}%
}}}}
\put(7966,-4471){\makebox(0,0)[lb]{\smash{{\SetFigFont{5}{6.0}{\rmdefault}{\mddefault}{\updefault}{\color[rgb]{0,0,0}$C$}%
}}}}
\put(901,-8131){\makebox(0,0)[lb]{\smash{{\SetFigFont{5}{6.0}{\rmdefault}{\mddefault}{\updefault}{\color[rgb]{0,0,0}$0$}%
}}}}
\put(10261,-1411){\makebox(0,0)[lb]{\smash{{\SetFigFont{5}{6.0}{\rmdefault}{\mddefault}{\updefault}{\color[rgb]{0,0,0}$\left(R,R \bar w_1\right)$}%
}}}}
\put(7831,-3076){\makebox(0,0)[lb]{\smash{{\SetFigFont{5}{6.0}{\rmdefault}{\mddefault}{\updefault}{\color[rgb]{0,0,0}$\left(\rho_1^\sharp, \eta_1^\sharp\right)$}%
}}}}
\end{picture}%

%% file: rhoeta_tau2222.pdftex_t
\begin{picture}(0,0)%
\includegraphics{rhoeta_tau2222.pdf}%
\end{picture}%
%
%
\setlength{\unitlength}{1066sp}%
\begingroup\makeatletter\ifx\SetFigFont\undefined%
\gdef\SetFigFont#1#2#3#4#5{%
  \reset@font\fontsize{#1}{#2pt}%
  \fontfamily{#3}\fontseries{#4}\fontshape{#5}%
  \selectfont}%
\fi\endgroup%
\begin{picture}(6135,8109)(193,-8173)
\put(6166,-1741){\makebox(0,0)[lb]{\smash{{\SetFigFont{5}{6.0}{\rmdefault}{\mddefault}{\updefault}{\color[rgb]{0,0,0}$t$}%
}}}}
\put(676,-7996){\makebox(0,0)[lb]{\smash{{\SetFigFont{5}{6.0}{\rmdefault}{\mddefault}{\updefault}{\color[rgb]{0,0,0}$x$}%
}}}}
\put(451,-5386){\makebox(0,0)[lb]{\smash{{\SetFigFont{5}{6.0}{\rmdefault}{\mddefault}{\updefault}{\color[rgb]{0,0,0}$(\bar \rho_{1}, \bar \eta_{1})$}%
}}}}
\put(1531,-7996){\makebox(0,0)[lb]{\smash{{\SetFigFont{5}{6.0}{\rmdefault}{\mddefault}{\updefault}{\color[rgb]{0,0,0}$(\bar \rho_{1}, \bar \eta_{1})$}%
}}}}
\put(3931,-3586){\makebox(0,0)[lb]{\smash{{\SetFigFont{5}{6.0}{\rmdefault}{\mddefault}{\updefault}{\color[rgb]{0,0,0}$\left(R,R \bar w_1\right)$}%
}}}}
\put(211,-7516){\makebox(0,0)[lb]{\smash{{\SetFigFont{5}{6.0}{\rmdefault}{\mddefault}{\updefault}{\color[rgb]{1,0,0}$I_1$}%
}}}}
\put(2311,-4726){\makebox(0,0)[lb]{\smash{{\SetFigFont{5}{6.0}{\rmdefault}{\mddefault}{\updefault}{\color[rgb]{0,0,0}$\left(\rho_1^\sharp, \eta_1^\sharp\right)$}%
}}}}
\put(3931,-7981){\makebox(0,0)[lb]{\smash{{\SetFigFont{5}{6.0}{\rmdefault}{\mddefault}{\updefault}{\color[rgb]{0,0,0}$\left(\rho_1^\sharp, \eta_1^\sharp\right)$}%
}}}}
\put(6151,-3061){\makebox(0,0)[lb]{\smash{{\SetFigFont{5}{6.0}{\rmdefault}{\mddefault}{\updefault}{\color[rgb]{0,0,0}$\tau_{1}^\ell+\tau_{2}^\ell$}%
}}}}
\put(6091,-7981){\makebox(0,0)[lb]{\smash{{\SetFigFont{5}{6.0}{\rmdefault}{\mddefault}{\updefault}{\color[rgb]{0,0,0}$\tau_{1}^\ell$}%
}}}}
\end{picture}%

%% file: rhoeta_tau3.pdftex_t
\begin{picture}(0,0)%
\includegraphics{rhoeta_tau3.pdf}%
\end{picture}%
%
%
\setlength{\unitlength}{829sp}%
\begingroup\makeatletter\ifx\SetFigFont\undefined%
\gdef\SetFigFont#1#2#3#4#5{%
  \reset@font\fontsize{#1}{#2pt}%
  \fontfamily{#3}\fontseries{#4}\fontshape{#5}%
  \selectfont}%
\fi\endgroup%
\begin{picture}(10377,8399)(886,-8398)
\put(1201,-556){\makebox(0,0)[lb]{\smash{{\SetFigFont{5}{6.0}{\rmdefault}{\mddefault}{\updefault}{\color[rgb]{0,0,0}$\eta$}%
}}}}
\put(5476,-4951){\makebox(0,0)[lb]{\smash{{\SetFigFont{5}{6.0}{\rmdefault}{\mddefault}{\updefault}{\color[rgb]{0,0,0}$F$}%
}}}}
\put(10966,-8056){\makebox(0,0)[lb]{\smash{{\SetFigFont{5}{6.0}{\rmdefault}{\mddefault}{\updefault}{\color[rgb]{0,0,0}$\rho$}%
}}}}
\put(9436,-8146){\makebox(0,0)[lb]{\smash{{\SetFigFont{5}{6.0}{\rmdefault}{\mddefault}{\updefault}{\color[rgb]{0,0,0}$R$}%
}}}}
\put(7126,-5311){\makebox(0,0)[lb]{\smash{{\SetFigFont{5}{6.0}{\rmdefault}{\mddefault}{\updefault}{\color[rgb]{0,0,0}$C$}%
}}}}
\put(901,-8131){\makebox(0,0)[lb]{\smash{{\SetFigFont{5}{6.0}{\rmdefault}{\mddefault}{\updefault}{\color[rgb]{0,0,0}$0$}%
}}}}
\put(10231,-1366){\makebox(0,0)[lb]{\smash{{\SetFigFont{5}{6.0}{\rmdefault}{\mddefault}{\updefault}{\color[rgb]{0,0,0}$(R,R \bar w_1)$}%
}}}}
\put(7531,-466){\makebox(0,0)[lb]{\smash{{\SetFigFont{5}{6.0}{\rmdefault}{\mddefault}{\updefault}{\color[rgb]{0,0,0}$\left(\rho_1^\sharp, \eta_1^\sharp\right)$}%
}}}}
\put(7666,-3796){\makebox(0,0)[lb]{\smash{{\SetFigFont{5}{6.0}{\rmdefault}{\mddefault}{\updefault}{\color[rgb]{0,0,0}$\left(\rho_2^\sharp, \eta_2^\sharp\right)$}%
}}}}
\end{picture}%

%% file: rhoeta_tau33.pdftex_t
\begin{picture}(0,0)%
\includegraphics{rhoeta_tau33.pdf}%
\end{picture}%
%
%
\setlength{\unitlength}{1026sp}%
\begingroup\makeatletter\ifx\SetFigFont\undefined%
\gdef\SetFigFont#1#2#3#4#5{%
  \reset@font\fontsize{#1}{#2pt}%
  \fontfamily{#3}\fontseries{#4}\fontshape{#5}%
  \selectfont}%
\fi\endgroup%
\begin{picture}(13440,8154)(136,-8248)
\put(4966,-1261){\makebox(0,0)[lb]{\smash{{\SetFigFont{5}{6.0}{\rmdefault}{\mddefault}{\updefault}{\color[rgb]{0,0,0}$t$}%
}}}}
\put(13096,-8101){\makebox(0,0)[lb]{\smash{{\SetFigFont{5}{6.0}{\rmdefault}{\mddefault}{\updefault}{\color[rgb]{0,0,0}$x$}%
}}}}
\put(3121,-6841){\makebox(0,0)[lb]{\smash{{\SetFigFont{5}{6.0}{\rmdefault}{\mddefault}{\updefault}{\color[rgb]{0,0,0}$(R,R \bar w_1)$}%
}}}}
\put(151,-7501){\makebox(0,0)[lb]{\smash{{\SetFigFont{5}{6.0}{\rmdefault}{\mddefault}{\updefault}{\color[rgb]{1,0,0}$I_1$}%
}}}}
\put(13561,-7546){\makebox(0,0)[lb]{\smash{{\SetFigFont{5}{6.0}{\rmdefault}{\mddefault}{\updefault}{\color[rgb]{1,0,0}$I_3$}%
}}}}
\put(12526,-6556){\makebox(0,0)[lb]{\smash{{\SetFigFont{5}{6.0}{\rmdefault}{\mddefault}{\updefault}{\color[rgb]{0,0,0}$(\bar \rho_{3}, \bar \eta_{3})$}%
}}}}
\put(6061,-4366){\makebox(0,0)[lb]{\smash{{\SetFigFont{5}{6.0}{\rmdefault}{\mddefault}{\updefault}{\color[rgb]{0,0,0}$\left(\rho_1^\sharp, \eta_1^\sharp\right)$}%
}}}}
\put(7831,-5446){\makebox(0,0)[lb]{\smash{{\SetFigFont{5}{6.0}{\rmdefault}{\mddefault}{\updefault}{\color[rgb]{0,0,0}$\left(\rho_2^\sharp, \eta_2^\sharp\right)$}%
}}}}
\put(6031,-8041){\makebox(0,0)[lb]{\smash{{\SetFigFont{5}{6.0}{\rmdefault}{\mddefault}{\updefault}{\color[rgb]{0,0,0}$\tau_{1}^\ell + \tau_{2}^\ell$}%
}}}}
\put(6121,-3166){\makebox(0,0)[lb]{\smash{{\SetFigFont{5}{6.0}{\rmdefault}{\mddefault}{\updefault}{\color[rgb]{0,0,0}$2\tau_{1}^\ell + \tau_{2}^\ell$}%
}}}}
\put(10291,-5761){\makebox(0,0)[lb]{\smash{{\SetFigFont{5}{6.0}{\rmdefault}{\mddefault}{\updefault}{\color[rgb]{0,0,0}$\left(\rho_1^\sharp, \eta_1^\sharp\right)$}%
}}}}
\put(391,-3976){\makebox(0,0)[lb]{\smash{{\SetFigFont{5}{6.0}{\rmdefault}{\mddefault}{\updefault}{\color[rgb]{0,0,0}$\left(\rho_1^\sharp, \eta_1^\sharp\right)$}%
}}}}
\end{picture}%

%% file: rhoeta_tau_star.pdftex_t
\begin{picture}(0,0)%
\includegraphics{rhoeta_tau_star.pdf}%
\end{picture}%
%
%
\setlength{\unitlength}{1342sp}%
\begingroup\makeatletter\ifx\SetFigFont\undefined%
\gdef\SetFigFont#1#2#3#4#5{%
  \reset@font\fontsize{#1}{#2pt}%
  \fontfamily{#3}\fontseries{#4}\fontshape{#5}%
  \selectfont}%
\fi\endgroup%
\begin{picture}(13335,8736)(1,-8308)
\put(901,-8131){\makebox(0,0)[lb]{\smash{{\SetFigFont{5}{6.0}{\rmdefault}{\mddefault}{\updefault}{\color[rgb]{0,0,0}$0$}%
}}}}
\put(1756,-76){\makebox(0,0)[lb]{\smash{{\SetFigFont{5}{6.0}{\rmdefault}{\mddefault}{\updefault}{\color[rgb]{0,0,0}$t$}%
}}}}
\put(13321,-7516){\makebox(0,0)[lb]{\smash{{\SetFigFont{5}{6.0}{\rmdefault}{\mddefault}{\updefault}{\color[rgb]{1,0,0}$I_3$}%
}}}}
\put( 31,-1156){\makebox(0,0)[lb]{\smash{{\SetFigFont{5}{6.0}{\rmdefault}{\mddefault}{\updefault}{\color[rgb]{0,0,0}$2\tau_{1}^\ell+2\tau_{2}^\ell$}%
}}}}
\put(1006,-5746){\makebox(0,0)[lb]{\smash{{\SetFigFont{5}{6.0}{\rmdefault}{\mddefault}{\updefault}{\color[rgb]{0,0,0}$\tau_{1}^\ell$}%
}}}}
\put( 76,-4186){\makebox(0,0)[lb]{\smash{{\SetFigFont{5}{6.0}{\rmdefault}{\mddefault}{\updefault}{\color[rgb]{0,0,0}$\tau_{1}^\ell+\tau_{2}^\ell$}%
}}}}
\put( 16,-2641){\makebox(0,0)[lb]{\smash{{\SetFigFont{5}{6.0}{\rmdefault}{\mddefault}{\updefault}{\color[rgb]{0,0,0}$2\tau_{1}^\ell+\tau_{2}^\ell$}%
}}}}
\put(2431,-5926){\makebox(0,0)[lb]{\smash{{\SetFigFont{5}{6.0}{\rmdefault}{\mddefault}{\updefault}{\color[rgb]{0,0,0}$\left(\rho_1^\sharp, \eta_1^\sharp\right)$}%
}}}}
\put(2356,254){\makebox(0,0)[lb]{\smash{{\SetFigFont{5}{6.0}{\rmdefault}{\mddefault}{\updefault}{\color[rgb]{0,0,0}$\left(\rho_1^\sharp, \eta_1^\sharp\right)$}%
}}}}
\put(2431,-2866){\makebox(0,0)[lb]{\smash{{\SetFigFont{5}{6.0}{\rmdefault}{\mddefault}{\updefault}{\color[rgb]{0,0,0}$\left(\rho_1^\sharp, \eta_1^\sharp\right)$}%
}}}}
\put(2386,-4381){\makebox(0,0)[lb]{\smash{{\SetFigFont{5}{6.0}{\rmdefault}{\mddefault}{\updefault}{\color[rgb]{0,0,0}$\left(\rho_2^\sharp, \eta_2^\sharp\right)$}%
}}}}
\put(2401,-1321){\makebox(0,0)[lb]{\smash{{\SetFigFont{5}{6.0}{\rmdefault}{\mddefault}{\updefault}{\color[rgb]{0,0,0}$\left(\rho_2^\sharp, \eta_2^\sharp\right)$}%
}}}}
\put(4381,-7021){\makebox(0,0)[lb]{\smash{{\SetFigFont{5}{6.0}{\rmdefault}{\mddefault}{\updefault}{\color[rgb]{0,0,0}$\left(\bar \rho_3, \bar \eta_3\right)$}%
}}}}
\put(4321,-3541){\makebox(0,0)[lb]{\smash{{\SetFigFont{5}{6.0}{\rmdefault}{\mddefault}{\updefault}{\color[rgb]{0,0,0}$A_2^\ell$}%
}}}}
\put(4336,-5161){\makebox(0,0)[lb]{\smash{{\SetFigFont{5}{6.0}{\rmdefault}{\mddefault}{\updefault}{\color[rgb]{0,0,0}$A_1^\ell$}%
}}}}
\put(4306,-1876){\makebox(0,0)[lb]{\smash{{\SetFigFont{5}{6.0}{\rmdefault}{\mddefault}{\updefault}{\color[rgb]{0,0,0}$A_1^\ell$}%
}}}}
\put(4306,-346){\makebox(0,0)[lb]{\smash{{\SetFigFont{5}{6.0}{\rmdefault}{\mddefault}{\updefault}{\color[rgb]{0,0,0}$A_2^\ell$}%
}}}}
\end{picture}%

%% file: rhov_RPRP.pdftex_t
\begin{picture}(0,0)%
\includegraphics{rhov_RPRP.pdf}%
\end{picture}%
%
%
\setlength{\unitlength}{868sp}%
\begingroup\makeatletter\ifx\SetFigFont\undefined%
\gdef\SetFigFont#1#2#3#4#5{%
  \reset@font\fontsize{#1}{#2pt}%
  \fontfamily{#3}\fontseries{#4}\fontshape{#5}%
  \selectfont}%
\fi\endgroup%
\begin{picture}(10377,8349)(886,-8398)
\put(11056,-8191){\makebox(0,0)[lb]{\smash{{\SetFigFont{5}{6.0}{\familydefault}{\mddefault}{\updefault}{\color[rgb]{0,0,0}$\rho$}%
}}}}
\put(2476,-4786){\makebox(0,0)[lb]{\smash{{\SetFigFont{5}{6.0}{\familydefault}{\mddefault}{\updefault}{\color[rgb]{0,0,0}$F$}%
}}}}
\put(5881,-2026){\makebox(0,0)[lb]{\smash{{\SetFigFont{5}{6.0}{\familydefault}{\mddefault}{\updefault}{\color[rgb]{0,0,0}$C$}%
}}}}
\put(1111,-8131){\makebox(0,0)[lb]{\smash{{\SetFigFont{5}{6.0}{\familydefault}{\mddefault}{\updefault}{\color[rgb]{0,0,0}$0$}%
}}}}
\put(9691,-8191){\makebox(0,0)[lb]{\smash{{\SetFigFont{5}{6.0}{\familydefault}{\mddefault}{\updefault}{\color[rgb]{0,0,0}$R$}%
}}}}
\put(901,-511){\makebox(0,0)[lb]{\smash{{\SetFigFont{5}{6.0}{\familydefault}{\mddefault}{\updefault}{\color[rgb]{0,0,0}$\rho v$}%
}}}}
\put(3241,-6196){\makebox(0,0)[lb]{\smash{{\SetFigFont{5}{6.0}{\rmdefault}{\mddefault}{\updefault}{\color[rgb]{0,0,0}$(\bar \rho_{1}, \bar \eta_{1})$}%
}}}}
\put(5581,-2641){\makebox(0,0)[lb]{\smash{{\SetFigFont{5}{6.0}{\rmdefault}{\mddefault}{\updefault}{\color[rgb]{0,0,0}$\left(\bar \rho_3, \bar \eta_3\right)$}%
}}}}
\put(6856,-4276){\makebox(0,0)[lb]{\smash{{\SetFigFont{5}{6.0}{\rmdefault}{\mddefault}{\updefault}{\color[rgb]{0,0,0}$\left(\rho_1^\sharp, \eta_1^\sharp\right)$}%
}}}}
\end{picture}%

%% file: rhoeta_tau11XT.pdftex_t
\begin{picture}(0,0)%
\includegraphics{rhoeta_tau11XT.pdf}%
\end{picture}%
%
%
\setlength{\unitlength}{829sp}%
\begingroup\makeatletter\ifx\SetFigFont\undefined%
\gdef\SetFigFont#1#2#3#4#5{%
  \reset@font\fontsize{#1}{#2pt}%
  \fontfamily{#3}\fontseries{#4}\fontshape{#5}%
  \selectfont}%
\fi\endgroup%
\begin{picture}(11445,8366)(136,-8190)
\put(1936,-8011){\makebox(0,0)[lb]{\smash{{\SetFigFont{5}{6.0}{\rmdefault}{\mddefault}{\updefault}{\color[rgb]{0,0,0}$(\bar \rho_{1}, \bar \eta_{1})$}%
}}}}
\put(4966,-76){\makebox(0,0)[lb]{\smash{{\SetFigFont{5}{6.0}{\rmdefault}{\mddefault}{\updefault}{\color[rgb]{0,0,0}$t$}%
}}}}
\put(10786,-8116){\makebox(0,0)[lb]{\smash{{\SetFigFont{5}{6.0}{\rmdefault}{\mddefault}{\updefault}{\color[rgb]{0,0,0}$x$}%
}}}}
\put(1486,-4291){\makebox(0,0)[lb]{\smash{{\SetFigFont{5}{6.0}{\rmdefault}{\mddefault}{\updefault}{\color[rgb]{0,0,0}$(\bar \rho_{1}, \bar \eta_{1})$}%
}}}}
\put(5101,-8026){\makebox(0,0)[lb]{\smash{{\SetFigFont{5}{6.0}{\rmdefault}{\mddefault}{\updefault}{\color[rgb]{0,0,0}$0$}%
}}}}
\put(7981,-5536){\makebox(0,0)[lb]{\smash{{\SetFigFont{5}{6.0}{\rmdefault}{\mddefault}{\updefault}{\color[rgb]{0,0,0}$\left(\bar \rho_3, \bar \eta_3\right)$}%
}}}}
\put(7471,-8041){\makebox(0,0)[lb]{\smash{{\SetFigFont{5}{6.0}{\rmdefault}{\mddefault}{\updefault}{\color[rgb]{0,0,0}$\left(\bar \rho_3, \bar \eta_3\right)$}%
}}}}
\put(151,-7501){\makebox(0,0)[lb]{\smash{{\SetFigFont{5}{6.0}{\rmdefault}{\mddefault}{\updefault}{\color[rgb]{1,0,0}$I_1$}%
}}}}
\put(11566,-7576){\makebox(0,0)[lb]{\smash{{\SetFigFont{5}{6.0}{\rmdefault}{\mddefault}{\updefault}{\color[rgb]{1,0,0}$I_3$}%
}}}}
\put(7111,-2701){\makebox(0,0)[lb]{\smash{{\SetFigFont{5}{6.0}{\rmdefault}{\mddefault}{\updefault}{\color[rgb]{0,0,0}$\left(\rho_1^\sharp, \eta_1^\sharp\right)$}%
}}}}
\put(4591,-3061){\makebox(0,0)[lb]{\smash{{\SetFigFont{5}{6.0}{\rmdefault}{\mddefault}{\updefault}{\color[rgb]{0,0,0}$\tau_{1}^\ell$}%
}}}}
\end{picture}%

%% file: rhov_RP.pdftex_t
\begin{picture}(0,0)%
\includegraphics{rhov_RP.pdf}%
\end{picture}%
%
%
\setlength{\unitlength}{868sp}%
\begingroup\makeatletter\ifx\SetFigFont\undefined%
\gdef\SetFigFont#1#2#3#4#5{%
  \reset@font\fontsize{#1}{#2pt}%
  \fontfamily{#3}\fontseries{#4}\fontshape{#5}%
  \selectfont}%
\fi\endgroup%
\begin{picture}(10377,8349)(886,-8398)
\put(11056,-8191){\makebox(0,0)[lb]{\smash{{\SetFigFont{5}{6.0}{\familydefault}{\mddefault}{\updefault}{\color[rgb]{0,0,0}$\rho$}%
}}}}
\put(2476,-4786){\makebox(0,0)[lb]{\smash{{\SetFigFont{5}{6.0}{\familydefault}{\mddefault}{\updefault}{\color[rgb]{0,0,0}$F$}%
}}}}
\put(5881,-2026){\makebox(0,0)[lb]{\smash{{\SetFigFont{5}{6.0}{\familydefault}{\mddefault}{\updefault}{\color[rgb]{0,0,0}$C$}%
}}}}
\put(1111,-8131){\makebox(0,0)[lb]{\smash{{\SetFigFont{5}{6.0}{\familydefault}{\mddefault}{\updefault}{\color[rgb]{0,0,0}$0$}%
}}}}
\put(9691,-8191){\makebox(0,0)[lb]{\smash{{\SetFigFont{5}{6.0}{\familydefault}{\mddefault}{\updefault}{\color[rgb]{0,0,0}$R$}%
}}}}
\put(901,-511){\makebox(0,0)[lb]{\smash{{\SetFigFont{5}{6.0}{\familydefault}{\mddefault}{\updefault}{\color[rgb]{0,0,0}$\rho v$}%
}}}}
\put(3241,-6196){\makebox(0,0)[lb]{\smash{{\SetFigFont{5}{6.0}{\rmdefault}{\mddefault}{\updefault}{\color[rgb]{0,0,0}$(\bar \rho_{1}, \bar \eta_{1})$}%
}}}}
\put(5446,-2686){\makebox(0,0)[lb]{\smash{{\SetFigFont{5}{6.0}{\rmdefault}{\mddefault}{\updefault}{\color[rgb]{0,0,0}$\left(\bar \rho_3, \bar \eta_3\right)$}%
}}}}
\put(6856,-4276){\makebox(0,0)[lb]{\smash{{\SetFigFont{5}{6.0}{\rmdefault}{\mddefault}{\updefault}{\color[rgb]{0,0,0}$\left(\rho_1^\sharp, \eta_1^\sharp\right)$}%
}}}}
\end{picture}%

%% file: rhoeta_tau11XTXT.pdftex_t
\begin{picture}(0,0)%
\includegraphics{rhoeta_tau11XTXT.pdf}%
\end{picture}%
%
%
\setlength{\unitlength}{829sp}%
\begingroup\makeatletter\ifx\SetFigFont\undefined%
\gdef\SetFigFont#1#2#3#4#5{%
  \reset@font\fontsize{#1}{#2pt}%
  \fontfamily{#3}\fontseries{#4}\fontshape{#5}%
  \selectfont}%
\fi\endgroup%
\begin{picture}(11445,8366)(136,-8190)
\put(1936,-8011){\makebox(0,0)[lb]{\smash{{\SetFigFont{5}{6.0}{\rmdefault}{\mddefault}{\updefault}{\color[rgb]{0,0,0}$(\bar \rho_{1}, \bar \eta_{1})$}%
}}}}
\put(4966,-76){\makebox(0,0)[lb]{\smash{{\SetFigFont{5}{6.0}{\rmdefault}{\mddefault}{\updefault}{\color[rgb]{0,0,0}$t$}%
}}}}
\put(10786,-8116){\makebox(0,0)[lb]{\smash{{\SetFigFont{5}{6.0}{\rmdefault}{\mddefault}{\updefault}{\color[rgb]{0,0,0}$x$}%
}}}}
\put(3511,-2011){\makebox(0,0)[lb]{\smash{{\SetFigFont{5}{6.0}{\rmdefault}{\mddefault}{\updefault}{\color[rgb]{0,0,0}$(\bar \rho_{1}, \bar \eta_{1})$}%
}}}}
\put(5101,-8026){\makebox(0,0)[lb]{\smash{{\SetFigFont{5}{6.0}{\rmdefault}{\mddefault}{\updefault}{\color[rgb]{0,0,0}$0$}%
}}}}
\put(7981,-5536){\makebox(0,0)[lb]{\smash{{\SetFigFont{5}{6.0}{\rmdefault}{\mddefault}{\updefault}{\color[rgb]{0,0,0}$\left(\bar \rho_3, \bar \eta_3\right)$}%
}}}}
\put(7471,-8041){\makebox(0,0)[lb]{\smash{{\SetFigFont{5}{6.0}{\rmdefault}{\mddefault}{\updefault}{\color[rgb]{0,0,0}$\left(\bar \rho_3, \bar \eta_3\right)$}%
}}}}
\put(151,-7501){\makebox(0,0)[lb]{\smash{{\SetFigFont{5}{6.0}{\rmdefault}{\mddefault}{\updefault}{\color[rgb]{1,0,0}$I_1$}%
}}}}
\put(11566,-7576){\makebox(0,0)[lb]{\smash{{\SetFigFont{5}{6.0}{\rmdefault}{\mddefault}{\updefault}{\color[rgb]{1,0,0}$I_3$}%
}}}}
\put(7966,-2251){\makebox(0,0)[lb]{\smash{{\SetFigFont{5}{6.0}{\rmdefault}{\mddefault}{\updefault}{\color[rgb]{0,0,0}$\left(\rho_1^\sharp, \eta_1^\sharp\right)$}%
}}}}
\put(4591,-3061){\makebox(0,0)[lb]{\smash{{\SetFigFont{5}{6.0}{\rmdefault}{\mddefault}{\updefault}{\color[rgb]{0,0,0}$\tau_{1}^\ell$}%
}}}}
\end{picture}%

%% file: part.pdftex_t
\begin{picture}(0,0)%
\includegraphics{part.pdf}%
\end{picture}%
%
%
\setlength{\unitlength}{1579sp}%
\begingroup\makeatletter\ifx\SetFigFont\undefined%
\gdef\SetFigFont#1#2#3#4#5{%
  \reset@font\fontsize{#1}{#2pt}%
  \fontfamily{#3}\fontseries{#4}\fontshape{#5}%
  \selectfont}%
\fi\endgroup%
\begin{picture}(7257,8366)(4324,-8190)
\put(4966,-76){\makebox(0,0)[lb]{\smash{{\SetFigFont{5}{6.0}{\rmdefault}{\mddefault}{\updefault}{\color[rgb]{0,0,0}$t$}%
}}}}
\put(10786,-8116){\makebox(0,0)[lb]{\smash{{\SetFigFont{5}{6.0}{\rmdefault}{\mddefault}{\updefault}{\color[rgb]{0,0,0}$x$}%
}}}}
\put(5101,-8026){\makebox(0,0)[lb]{\smash{{\SetFigFont{5}{6.0}{\rmdefault}{\mddefault}{\updefault}{\color[rgb]{0,0,0}$0$}%
}}}}
\put(7981,-5536){\makebox(0,0)[lb]{\smash{{\SetFigFont{5}{6.0}{\rmdefault}{\mddefault}{\updefault}{\color[rgb]{0,0,0}$\left(\bar \rho_3, \bar \eta_3\right)$}%
}}}}
\put(11566,-7576){\makebox(0,0)[lb]{\smash{{\SetFigFont{5}{6.0}{\rmdefault}{\mddefault}{\updefault}{\color[rgb]{1,0,0}$I_3$}%
}}}}
\put(7996,-2761){\makebox(0,0)[lb]{\smash{{\SetFigFont{5}{6.0}{\rmdefault}{\mddefault}{\updefault}{\color[rgb]{0,0,0}$\left(\rho_1^\sharp, \eta_1^\sharp\right)$}%
}}}}
\put(4696,-3136){\makebox(0,0)[lb]{\smash{{\SetFigFont{5}{6.0}{\rmdefault}{\mddefault}{\updefault}{\color[rgb]{0,0,0}$\tau_{1}^\ell$}%
}}}}
\put(5776,-4171){\makebox(0,0)[lb]{\smash{{\SetFigFont{5}{6.0}{\rmdefault}{\mddefault}{\updefault}{\color[rgb]{0,0,0}$(\bar \rho_{1}, \bar \eta_{1})$}%
}}}}
\put(4666,-1291){\rotatebox{1.0}{\makebox(0,0)[lb]{\smash{{\SetFigFont{5}{6.0}{\rmdefault}{\mddefault}{\updefault}{\color[rgb]{0,0,0}$t=\bar t_2$}%
}}}}}
\put(4726,-2671){\rotatebox{1.0}{\makebox(0,0)[lb]{\smash{{\SetFigFont{5}{6.0}{\rmdefault}{\mddefault}{\updefault}{\color[rgb]{0,0,0}$t=\bar t_1$}%
}}}}}
\put(7456,-8041){\rotatebox{1.0}{\makebox(0,0)[lb]{\smash{{\SetFigFont{5}{6.0}{\rmdefault}{\mddefault}{\updefault}{\color[rgb]{0,0,0}$t=\bar x_1$}%
}}}}}
\put(5806,-2446){\makebox(0,0)[lb]{\smash{{\SetFigFont{5}{6.0}{\rmdefault}{\mddefault}{\updefault}{\color[rgb]{0,0,0}$\left(\rho_2^\flat, \eta_2^\flat\right)$}%
}}}}
\put(6766,-1081){\makebox(0,0)[lb]{\smash{{\SetFigFont{5}{6.0}{\rmdefault}{\mddefault}{\updefault}{\color[rgb]{0,0,0}$\left(\rho_2^\sharp, \eta_2^\sharp\right)$}%
}}}}
\end{picture}%

%% file: rhoeta_tau1_Lemma3.pdftex_t
\begin{picture}(0,0)%
\includegraphics{rhoeta_tau1_Lemma3.pdf}%
\end{picture}%
%
%
\setlength{\unitlength}{868sp}%
\begingroup\makeatletter\ifx\SetFigFont\undefined%
\gdef\SetFigFont#1#2#3#4#5{%
  \reset@font\fontsize{#1}{#2pt}%
  \fontfamily{#3}\fontseries{#4}\fontshape{#5}%
  \selectfont}%
\fi\endgroup%
\begin{picture}(10374,8399)(889,-8398)
\put(1201,-556){\makebox(0,0)[lb]{\smash{{\SetFigFont{5}{6.0}{\rmdefault}{\mddefault}{\updefault}{\color[rgb]{0,0,0}$\eta$}%
}}}}
\put(5446,-5536){\makebox(0,0)[lb]{\smash{{\SetFigFont{5}{6.0}{\rmdefault}{\mddefault}{\updefault}{\color[rgb]{0,0,0}$F$}%
}}}}
\put(10966,-8056){\makebox(0,0)[lb]{\smash{{\SetFigFont{5}{6.0}{\rmdefault}{\mddefault}{\updefault}{\color[rgb]{0,0,0}$\rho$}%
}}}}
\put(9436,-8146){\makebox(0,0)[lb]{\smash{{\SetFigFont{5}{6.0}{\rmdefault}{\mddefault}{\updefault}{\color[rgb]{0,0,0}$R$}%
}}}}
\put(8236,-3271){\makebox(0,0)[lb]{\smash{{\SetFigFont{5}{6.0}{\rmdefault}{\mddefault}{\updefault}{\color[rgb]{0,0,0}$C$}%
}}}}
\put(4636,-7111){\makebox(0,0)[lb]{\smash{{\SetFigFont{5}{6.0}{\rmdefault}{\mddefault}{\updefault}{\color[rgb]{0,0,0}$(\bar \rho_{1}, \bar \eta_{1})$}%
}}}}
\put(931,-8056){\makebox(0,0)[lb]{\smash{{\SetFigFont{5}{6.0}{\rmdefault}{\mddefault}{\updefault}{\color[rgb]{0,0,0}$0$}%
}}}}
\put(4456,-3406){\makebox(0,0)[lb]{\smash{{\SetFigFont{5}{6.0}{\rmdefault}{\mddefault}{\updefault}{\color[rgb]{0,0,0}$\left(\bar \rho_3, \bar \eta_3\right)$}%
}}}}
\end{picture}%

%% file: rhoeta_tau11_Lemma3.pdftex_t
\begin{picture}(0,0)%
\includegraphics{rhoeta_tau11_Lemma3.pdf}%
\end{picture}%
%
%
\setlength{\unitlength}{868sp}%
\begingroup\makeatletter\ifx\SetFigFont\undefined%
\gdef\SetFigFont#1#2#3#4#5{%
  \reset@font\fontsize{#1}{#2pt}%
  \fontfamily{#3}\fontseries{#4}\fontshape{#5}%
  \selectfont}%
\fi\endgroup%
\begin{picture}(11445,8366)(136,-8190)
\put(1936,-8011){\makebox(0,0)[lb]{\smash{{\SetFigFont{5}{6.0}{\rmdefault}{\mddefault}{\updefault}{\color[rgb]{0,0,0}$(\bar \rho_{1}, \bar \eta_{1})$}%
}}}}
\put(4966,-76){\makebox(0,0)[lb]{\smash{{\SetFigFont{5}{6.0}{\rmdefault}{\mddefault}{\updefault}{\color[rgb]{0,0,0}$t$}%
}}}}
\put(10786,-8116){\makebox(0,0)[lb]{\smash{{\SetFigFont{5}{6.0}{\rmdefault}{\mddefault}{\updefault}{\color[rgb]{0,0,0}$x$}%
}}}}
\put(6811,-2431){\makebox(0,0)[lb]{\smash{{\SetFigFont{5}{6.0}{\rmdefault}{\mddefault}{\updefault}{\color[rgb]{0,0,0}$(\bar \rho_{1}, \bar \eta_{1})$}%
}}}}
\put(5101,-8026){\makebox(0,0)[lb]{\smash{{\SetFigFont{5}{6.0}{\rmdefault}{\mddefault}{\updefault}{\color[rgb]{0,0,0}$0$}%
}}}}
\put(8626,-6331){\makebox(0,0)[lb]{\smash{{\SetFigFont{5}{6.0}{\rmdefault}{\mddefault}{\updefault}{\color[rgb]{0,0,0}$\left(\bar \rho_3, \bar \eta_3\right)$}%
}}}}
\put(7471,-8041){\makebox(0,0)[lb]{\smash{{\SetFigFont{5}{6.0}{\rmdefault}{\mddefault}{\updefault}{\color[rgb]{0,0,0}$\left(\bar \rho_3, \bar \eta_3\right)$}%
}}}}
\put(151,-7501){\makebox(0,0)[lb]{\smash{{\SetFigFont{5}{6.0}{\rmdefault}{\mddefault}{\updefault}{\color[rgb]{1,0,0}$I_1$}%
}}}}
\put(11566,-7576){\makebox(0,0)[lb]{\smash{{\SetFigFont{5}{6.0}{\rmdefault}{\mddefault}{\updefault}{\color[rgb]{1,0,0}$I_3$}%
}}}}
\put(4591,-3061){\makebox(0,0)[lb]{\smash{{\SetFigFont{5}{6.0}{\rmdefault}{\mddefault}{\updefault}{\color[rgb]{0,0,0}$\tau_{1}^\ell$}%
}}}}
\end{picture}%

%% file: rhoeta_tau111_Lemma3.pdftex_t
\begin{picture}(0,0)%
\includegraphics{rhoeta_tau111_Lemma3.pdf}%
\end{picture}%
%
%
\setlength{\unitlength}{829sp}%
\begingroup\makeatletter\ifx\SetFigFont\undefined%
\gdef\SetFigFont#1#2#3#4#5{%
  \reset@font\fontsize{#1}{#2pt}%
  \fontfamily{#3}\fontseries{#4}\fontshape{#5}%
  \selectfont}%
\fi\endgroup%
\begin{picture}(10377,8399)(886,-8398)
\put(1201,-556){\makebox(0,0)[lb]{\smash{{\SetFigFont{5}{6.0}{\rmdefault}{\mddefault}{\updefault}{\color[rgb]{0,0,0}$\eta$}%
}}}}
\put(5071,-5926){\makebox(0,0)[lb]{\smash{{\SetFigFont{5}{6.0}{\rmdefault}{\mddefault}{\updefault}{\color[rgb]{0,0,0}$F$}%
}}}}
\put(10966,-8056){\makebox(0,0)[lb]{\smash{{\SetFigFont{5}{6.0}{\rmdefault}{\mddefault}{\updefault}{\color[rgb]{0,0,0}$\rho$}%
}}}}
\put(9436,-8146){\makebox(0,0)[lb]{\smash{{\SetFigFont{5}{6.0}{\rmdefault}{\mddefault}{\updefault}{\color[rgb]{0,0,0}$R$}%
}}}}
\put(7351,-5176){\makebox(0,0)[lb]{\smash{{\SetFigFont{5}{6.0}{\rmdefault}{\mddefault}{\updefault}{\color[rgb]{0,0,0}$C$}%
}}}}
\put(901,-8131){\makebox(0,0)[lb]{\smash{{\SetFigFont{5}{6.0}{\rmdefault}{\mddefault}{\updefault}{\color[rgb]{0,0,0}$0$}%
}}}}
\put(7831,-3826){\makebox(0,0)[lb]{\smash{{\SetFigFont{5}{6.0}{\rmdefault}{\mddefault}{\updefault}{\color[rgb]{0,0,0}$(\bar \rho_{2}, \bar \eta_{2})$}%
}}}}
\put(10321,-2536){\makebox(0,0)[lb]{\smash{{\SetFigFont{5}{6.0}{\rmdefault}{\mddefault}{\updefault}{\color[rgb]{0,0,0}$\left(R,R \bar w_2\right)$}%
}}}}
\end{picture}%

%% file: rhoeta_tau2_Lemma3.pdftex_t
\begin{picture}(0,0)%
\includegraphics{rhoeta_tau2_Lemma3.pdf}%
\end{picture}%
%
%
\setlength{\unitlength}{868sp}%
\begingroup\makeatletter\ifx\SetFigFont\undefined%
\gdef\SetFigFont#1#2#3#4#5{%
  \reset@font\fontsize{#1}{#2pt}%
  \fontfamily{#3}\fontseries{#4}\fontshape{#5}%
  \selectfont}%
\fi\endgroup%
\begin{picture}(10377,8399)(886,-8398)
\put(1201,-556){\makebox(0,0)[lb]{\smash{{\SetFigFont{5}{6.0}{\rmdefault}{\mddefault}{\updefault}{\color[rgb]{0,0,0}$\eta$}%
}}}}
\put(4516,-5551){\makebox(0,0)[lb]{\smash{{\SetFigFont{5}{6.0}{\rmdefault}{\mddefault}{\updefault}{\color[rgb]{0,0,0}$F$}%
}}}}
\put(10966,-8056){\makebox(0,0)[lb]{\smash{{\SetFigFont{5}{6.0}{\rmdefault}{\mddefault}{\updefault}{\color[rgb]{0,0,0}$\rho$}%
}}}}
\put(9436,-8146){\makebox(0,0)[lb]{\smash{{\SetFigFont{5}{6.0}{\rmdefault}{\mddefault}{\updefault}{\color[rgb]{0,0,0}$R$}%
}}}}
\put(7861,-3076){\makebox(0,0)[lb]{\smash{{\SetFigFont{5}{6.0}{\rmdefault}{\mddefault}{\updefault}{\color[rgb]{0,0,0}$C$}%
}}}}
\put(901,-8131){\makebox(0,0)[lb]{\smash{{\SetFigFont{5}{6.0}{\rmdefault}{\mddefault}{\updefault}{\color[rgb]{0,0,0}$0$}%
}}}}
\put(10231,-2806){\makebox(0,0)[lb]{\smash{{\SetFigFont{5}{6.0}{\rmdefault}{\mddefault}{\updefault}{\color[rgb]{0,0,0}$(R, R \bar w_2)$}%
}}}}
\put(6601,-4996){\makebox(0,0)[lb]{\smash{{\SetFigFont{5}{6.0}{\rmdefault}{\mddefault}{\updefault}{\color[rgb]{0,0,0}$\left(\rho_2^\flat, \eta_2^\flat\right)$}%
}}}}
\put(4756,-7096){\makebox(0,0)[lb]{\smash{{\SetFigFont{5}{6.0}{\rmdefault}{\mddefault}{\updefault}{\color[rgb]{0,0,0}$(\bar \rho_{1}, \bar \eta_{1})$}%
}}}}
\end{picture}%

%% file: rhoeta_tau22_Lemma3.pdftex_t
\begin{picture}(0,0)%
\includegraphics{rhoeta_tau22_Lemma3.pdf}%
\end{picture}%
%
%
\setlength{\unitlength}{1026sp}%
\begingroup\makeatletter\ifx\SetFigFont\undefined%
\gdef\SetFigFont#1#2#3#4#5{%
  \reset@font\fontsize{#1}{#2pt}%
  \fontfamily{#3}\fontseries{#4}\fontshape{#5}%
  \selectfont}%
\fi\endgroup%
\begin{picture}(11340,8154)(211,-8248)
\put(6661,-7981){\makebox(0,0)[lb]{\smash{{\SetFigFont{5}{6.0}{\rmdefault}{\mddefault}{\updefault}{\color[rgb]{0,0,0}$(\bar \rho_{1}, \bar \eta_{1})$}%
}}}}
\put(4966,-1261){\makebox(0,0)[lb]{\smash{{\SetFigFont{5}{6.0}{\rmdefault}{\mddefault}{\updefault}{\color[rgb]{0,0,0}$t$}%
}}}}
\put(10921,-8041){\makebox(0,0)[lb]{\smash{{\SetFigFont{5}{6.0}{\rmdefault}{\mddefault}{\updefault}{\color[rgb]{0,0,0}$x$}%
}}}}
\put(1066,-7921){\makebox(0,0)[lb]{\smash{{\SetFigFont{5}{6.0}{\rmdefault}{\mddefault}{\updefault}{\color[rgb]{0,0,0}$(\bar \rho_{2}, \bar \eta_{2})$}%
}}}}
\put(7501,-6331){\makebox(0,0)[lb]{\smash{{\SetFigFont{5}{6.0}{\rmdefault}{\mddefault}{\updefault}{\color[rgb]{0,0,0}$(\bar \rho_{1}, \bar \eta_{1})$}%
}}}}
\put(2581,-6166){\makebox(0,0)[lb]{\smash{{\SetFigFont{5}{6.0}{\rmdefault}{\mddefault}{\updefault}{\color[rgb]{0,0,0}$(R, R \bar w_2)$}%
}}}}
\put(3061,-7951){\makebox(0,0)[lb]{\smash{{\SetFigFont{5}{6.0}{\rmdefault}{\mddefault}{\updefault}{\color[rgb]{0,0,0}$(R, R \bar w_2)$}%
}}}}
\put(226,-7531){\makebox(0,0)[lb]{\smash{{\SetFigFont{5}{6.0}{\rmdefault}{\mddefault}{\updefault}{\color[rgb]{1,0,0}$I_2$}%
}}}}
\put(11536,-7561){\makebox(0,0)[lb]{\smash{{\SetFigFont{5}{6.0}{\rmdefault}{\mddefault}{\updefault}{\color[rgb]{1,0,0}$I_3$}%
}}}}
\put(8641,-8011){\makebox(0,0)[lb]{\smash{{\SetFigFont{5}{6.0}{\rmdefault}{\mddefault}{\updefault}{\color[rgb]{0,0,0}$(\bar \rho_{3}, \bar \eta_{3})$}%
}}}}
\put(5941,-4171){\makebox(0,0)[lb]{\smash{{\SetFigFont{5}{6.0}{\rmdefault}{\mddefault}{\updefault}{\color[rgb]{0,0,0}$\left(\rho_2^\flat, \eta_2^\flat\right)$}%
}}}}
\put(5776,-7936){\makebox(0,0)[lb]{\smash{{\SetFigFont{5}{6.0}{\rmdefault}{\mddefault}{\updefault}{\color[rgb]{0,0,0}$\tau_{1}^\ell$}%
}}}}
\put(6121,-3136){\makebox(0,0)[lb]{\smash{{\SetFigFont{5}{6.0}{\rmdefault}{\mddefault}{\updefault}{\color[rgb]{0,0,0}$\tau_{1}^\ell+\tau_{2}^\ell$}%
}}}}
\end{picture}%

%% file: rhoeta_tau222_Lemma3.pdftex_t
\begin{picture}(0,0)%
\includegraphics{rhoeta_tau222_Lemma3.pdf}%
\end{picture}%
%
%
\setlength{\unitlength}{868sp}%
\begingroup\makeatletter\ifx\SetFigFont\undefined%
\gdef\SetFigFont#1#2#3#4#5{%
  \reset@font\fontsize{#1}{#2pt}%
  \fontfamily{#3}\fontseries{#4}\fontshape{#5}%
  \selectfont}%
\fi\endgroup%
\begin{picture}(10377,8399)(886,-8398)
\put(1201,-556){\makebox(0,0)[lb]{\smash{{\SetFigFont{5}{6.0}{\rmdefault}{\mddefault}{\updefault}{\color[rgb]{0,0,0}$\eta$}%
}}}}
\put(5536,-5131){\makebox(0,0)[lb]{\smash{{\SetFigFont{5}{6.0}{\rmdefault}{\mddefault}{\updefault}{\color[rgb]{0,0,0}$F$}%
}}}}
\put(10966,-8056){\makebox(0,0)[lb]{\smash{{\SetFigFont{5}{6.0}{\rmdefault}{\mddefault}{\updefault}{\color[rgb]{0,0,0}$\rho$}%
}}}}
\put(9436,-8146){\makebox(0,0)[lb]{\smash{{\SetFigFont{5}{6.0}{\rmdefault}{\mddefault}{\updefault}{\color[rgb]{0,0,0}$R$}%
}}}}
\put(8776,-2581){\makebox(0,0)[lb]{\smash{{\SetFigFont{5}{6.0}{\rmdefault}{\mddefault}{\updefault}{\color[rgb]{0,0,0}$C$}%
}}}}
\put(901,-8131){\makebox(0,0)[lb]{\smash{{\SetFigFont{5}{6.0}{\rmdefault}{\mddefault}{\updefault}{\color[rgb]{0,0,0}$0$}%
}}}}
\put(10276,-4276){\makebox(0,0)[lb]{\smash{{\SetFigFont{5}{6.0}{\rmdefault}{\mddefault}{\updefault}{\color[rgb]{0,0,0}$\left(R,R \bar w_1\right)$}%
}}}}
\put(4801,-7081){\makebox(0,0)[lb]{\smash{{\SetFigFont{5}{6.0}{\rmdefault}{\mddefault}{\updefault}{\color[rgb]{0,0,0}$(\bar \rho_{1}, \bar \eta_{1})$}%
}}}}
\end{picture}%

%% file: rhoeta_tau2222_Lemma3.pdftex_t
\begin{picture}(0,0)%
\includegraphics{rhoeta_tau2222_Lemma3.pdf}%
\end{picture}%
%
%
\setlength{\unitlength}{1066sp}%
\begingroup\makeatletter\ifx\SetFigFont\undefined%
\gdef\SetFigFont#1#2#3#4#5{%
  \reset@font\fontsize{#1}{#2pt}%
  \fontfamily{#3}\fontseries{#4}\fontshape{#5}%
  \selectfont}%
\fi\endgroup%
\begin{picture}(6132,8109)(196,-8173)
\put(6166,-1741){\makebox(0,0)[lb]{\smash{{\SetFigFont{5}{6.0}{\rmdefault}{\mddefault}{\updefault}{\color[rgb]{0,0,0}$t$}%
}}}}
\put(676,-7996){\makebox(0,0)[lb]{\smash{{\SetFigFont{5}{6.0}{\rmdefault}{\mddefault}{\updefault}{\color[rgb]{0,0,0}$x$}%
}}}}
\put(1606,-5371){\makebox(0,0)[lb]{\smash{{\SetFigFont{5}{6.0}{\rmdefault}{\mddefault}{\updefault}{\color[rgb]{0,0,0}$(\bar \rho_{1}, \bar \eta_{1})$}%
}}}}
\put(1531,-7996){\makebox(0,0)[lb]{\smash{{\SetFigFont{5}{6.0}{\rmdefault}{\mddefault}{\updefault}{\color[rgb]{0,0,0}$(\bar \rho_{1}, \bar \eta_{1})$}%
}}}}
\put(3301,-3586){\makebox(0,0)[lb]{\smash{{\SetFigFont{5}{6.0}{\rmdefault}{\mddefault}{\updefault}{\color[rgb]{0,0,0}$\left(R,R \bar w_1\right)$}%
}}}}
\put(211,-7516){\makebox(0,0)[lb]{\smash{{\SetFigFont{5}{6.0}{\rmdefault}{\mddefault}{\updefault}{\color[rgb]{1,0,0}$I_1$}%
}}}}
\put(6151,-3061){\makebox(0,0)[lb]{\smash{{\SetFigFont{5}{6.0}{\rmdefault}{\mddefault}{\updefault}{\color[rgb]{0,0,0}$\tau_{1}^\ell+\tau_{2}^\ell$}%
}}}}
\put(6091,-7981){\makebox(0,0)[lb]{\smash{{\SetFigFont{5}{6.0}{\rmdefault}{\mddefault}{\updefault}{\color[rgb]{0,0,0}$\tau_{1}^\ell$}%
}}}}
\end{picture}%

%% file: rhoeta_tau3_Lemma3.pdftex_t
\begin{picture}(0,0)%
\includegraphics{rhoeta_tau3_Lemma3.pdf}%
\end{picture}%
%
%
\setlength{\unitlength}{829sp}%
\begingroup\makeatletter\ifx\SetFigFont\undefined%
\gdef\SetFigFont#1#2#3#4#5{%
  \reset@font\fontsize{#1}{#2pt}%
  \fontfamily{#3}\fontseries{#4}\fontshape{#5}%
  \selectfont}%
\fi\endgroup%
\begin{picture}(10377,8399)(886,-8398)
\put(1201,-556){\makebox(0,0)[lb]{\smash{{\SetFigFont{5}{6.0}{\rmdefault}{\mddefault}{\updefault}{\color[rgb]{0,0,0}$\eta$}%
}}}}
\put(5476,-4951){\makebox(0,0)[lb]{\smash{{\SetFigFont{5}{6.0}{\rmdefault}{\mddefault}{\updefault}{\color[rgb]{0,0,0}$F$}%
}}}}
\put(10966,-8056){\makebox(0,0)[lb]{\smash{{\SetFigFont{5}{6.0}{\rmdefault}{\mddefault}{\updefault}{\color[rgb]{0,0,0}$\rho$}%
}}}}
\put(9436,-8146){\makebox(0,0)[lb]{\smash{{\SetFigFont{5}{6.0}{\rmdefault}{\mddefault}{\updefault}{\color[rgb]{0,0,0}$R$}%
}}}}
\put(7876,-3301){\makebox(0,0)[lb]{\smash{{\SetFigFont{5}{6.0}{\rmdefault}{\mddefault}{\updefault}{\color[rgb]{0,0,0}$C$}%
}}}}
\put(901,-8131){\makebox(0,0)[lb]{\smash{{\SetFigFont{5}{6.0}{\rmdefault}{\mddefault}{\updefault}{\color[rgb]{0,0,0}$0$}%
}}}}
\put(10261,-4276){\makebox(0,0)[lb]{\smash{{\SetFigFont{5}{6.0}{\rmdefault}{\mddefault}{\updefault}{\color[rgb]{0,0,0}$(R,R \bar w_1)$}%
}}}}
\put(5761,-6646){\makebox(0,0)[lb]{\smash{{\SetFigFont{5}{6.0}{\rmdefault}{\mddefault}{\updefault}{\color[rgb]{0,0,0}$\left(\rho_1^\flat, \eta_1^\flat\right)$}%
}}}}
\put(7021,-4051){\makebox(0,0)[lb]{\smash{{\SetFigFont{5}{6.0}{\rmdefault}{\mddefault}{\updefault}{\color[rgb]{0,0,0}$\left(\rho_2^\flat, \eta_2^\flat\right)$}%
}}}}
\end{picture}%

%% file: rhoeta_tau33_Lemma3.pdftex_t
\begin{picture}(0,0)%
\includegraphics{rhoeta_tau33_Lemma3.pdf}%
\end{picture}%
%
%
\setlength{\unitlength}{987sp}%
\begingroup\makeatletter\ifx\SetFigFont\undefined%
\gdef\SetFigFont#1#2#3#4#5{%
  \reset@font\fontsize{#1}{#2pt}%
  \fontfamily{#3}\fontseries{#4}\fontshape{#5}%
  \selectfont}%
\fi\endgroup%
\begin{picture}(13440,8154)(136,-8248)
\put(4966,-1261){\makebox(0,0)[lb]{\smash{{\SetFigFont{5}{6.0}{\rmdefault}{\mddefault}{\updefault}{\color[rgb]{0,0,0}$t$}%
}}}}
\put(13096,-8101){\makebox(0,0)[lb]{\smash{{\SetFigFont{5}{6.0}{\rmdefault}{\mddefault}{\updefault}{\color[rgb]{0,0,0}$x$}%
}}}}
\put(3091,-6946){\makebox(0,0)[lb]{\smash{{\SetFigFont{5}{6.0}{\rmdefault}{\mddefault}{\updefault}{\color[rgb]{0,0,0}$(R,R \bar w_1)$}%
}}}}
\put(151,-7501){\makebox(0,0)[lb]{\smash{{\SetFigFont{5}{6.0}{\rmdefault}{\mddefault}{\updefault}{\color[rgb]{1,0,0}$I_1$}%
}}}}
\put(13561,-7546){\makebox(0,0)[lb]{\smash{{\SetFigFont{5}{6.0}{\rmdefault}{\mddefault}{\updefault}{\color[rgb]{1,0,0}$I_3$}%
}}}}
\put(12256,-6961){\makebox(0,0)[lb]{\smash{{\SetFigFont{5}{6.0}{\rmdefault}{\mddefault}{\updefault}{\color[rgb]{0,0,0}$(\bar \rho_{3}, \bar \eta_{3})$}%
}}}}
\put(6031,-8041){\makebox(0,0)[lb]{\smash{{\SetFigFont{5}{6.0}{\rmdefault}{\mddefault}{\updefault}{\color[rgb]{0,0,0}$\tau_{1}^\ell + \tau_{2}^\ell$}%
}}}}
\put(6121,-3166){\makebox(0,0)[lb]{\smash{{\SetFigFont{5}{6.0}{\rmdefault}{\mddefault}{\updefault}{\color[rgb]{0,0,0}$2\tau_{1}^\ell + \tau_{2}^\ell$}%
}}}}
\put(6301,-4126){\makebox(0,0)[lb]{\smash{{\SetFigFont{5}{6.0}{\rmdefault}{\mddefault}{\updefault}{\color[rgb]{0,0,0}$\left(\rho_1^\flat, \eta_1^\flat\right)$}%
}}}}
\put(7606,-6016){\makebox(0,0)[lb]{\smash{{\SetFigFont{5}{6.0}{\rmdefault}{\mddefault}{\updefault}{\color[rgb]{0,0,0}$\left(\rho_2^\flat, \eta_2^\flat\right)$}%
}}}}
\put(481,-5446){\makebox(0,0)[lb]{\smash{{\SetFigFont{5}{6.0}{\rmdefault}{\mddefault}{\updefault}{\color[rgb]{0,0,0}$(\bar \rho_{1}, \bar \eta_{1})$}%
}}}}
\put(11611,-4951){\makebox(0,0)[lb]{\smash{{\SetFigFont{5}{6.0}{\rmdefault}{\mddefault}{\updefault}{\color[rgb]{0,0,0}$(\bar \rho_{1}, \bar \eta_{1})$}%
}}}}
\end{picture}%

%% file: rhoeta_tau_star_Lemma3.pdftex_t
\begin{picture}(0,0)%
\includegraphics{rhoeta_tau_star_Lemma3.pdf}%
\end{picture}%
%
%
\setlength{\unitlength}{1342sp}%
\begingroup\makeatletter\ifx\SetFigFont\undefined%
\gdef\SetFigFont#1#2#3#4#5{%
  \reset@font\fontsize{#1}{#2pt}%
  \fontfamily{#3}\fontseries{#4}\fontshape{#5}%
  \selectfont}%
\fi\endgroup%
\begin{picture}(13335,8676)(1,-8308)
\put(901,-8131){\makebox(0,0)[lb]{\smash{{\SetFigFont{5}{6.0}{\rmdefault}{\mddefault}{\updefault}{\color[rgb]{0,0,0}$0$}%
}}}}
\put(1756,-76){\makebox(0,0)[lb]{\smash{{\SetFigFont{5}{6.0}{\rmdefault}{\mddefault}{\updefault}{\color[rgb]{0,0,0}$t$}%
}}}}
\put(13321,-7516){\makebox(0,0)[lb]{\smash{{\SetFigFont{5}{6.0}{\rmdefault}{\mddefault}{\updefault}{\color[rgb]{1,0,0}$I_3$}%
}}}}
\put( 31,-1156){\makebox(0,0)[lb]{\smash{{\SetFigFont{5}{6.0}{\rmdefault}{\mddefault}{\updefault}{\color[rgb]{0,0,0}$2\tau_{1}^\ell+2\tau_{2}^\ell$}%
}}}}
\put(1006,-5746){\makebox(0,0)[lb]{\smash{{\SetFigFont{5}{6.0}{\rmdefault}{\mddefault}{\updefault}{\color[rgb]{0,0,0}$\tau_{1}^\ell$}%
}}}}
\put( 76,-4186){\makebox(0,0)[lb]{\smash{{\SetFigFont{5}{6.0}{\rmdefault}{\mddefault}{\updefault}{\color[rgb]{0,0,0}$\tau_{1}^\ell+\tau_{2}^\ell$}%
}}}}
\put( 16,-2641){\makebox(0,0)[lb]{\smash{{\SetFigFont{5}{6.0}{\rmdefault}{\mddefault}{\updefault}{\color[rgb]{0,0,0}$2\tau_{1}^\ell+\tau_{2}^\ell$}%
}}}}
\put(4381,-7021){\makebox(0,0)[lb]{\smash{{\SetFigFont{5}{6.0}{\rmdefault}{\mddefault}{\updefault}{\color[rgb]{0,0,0}$\left(\bar \rho_3, \bar \eta_3\right)$}%
}}}}
\put(2476,-6016){\makebox(0,0)[lb]{\smash{{\SetFigFont{5}{6.0}{\rmdefault}{\mddefault}{\updefault}{\color[rgb]{0,0,0}$\left(\bar \rho_{1}, \bar \eta_1\right)$}%
}}}}
\put(2416,-2926){\makebox(0,0)[lb]{\smash{{\SetFigFont{5}{6.0}{\rmdefault}{\mddefault}{\updefault}{\color[rgb]{0,0,0}$\left(\rho_1^\flat, \eta_1^\flat\right)$}%
}}}}
\put(2371,194){\makebox(0,0)[lb]{\smash{{\SetFigFont{5}{6.0}{\rmdefault}{\mddefault}{\updefault}{\color[rgb]{0,0,0}$\left(\rho_2^\sharp, \eta_2^\sharp\right)$}%
}}}}
\put(2446,-4381){\makebox(0,0)[lb]{\smash{{\SetFigFont{5}{6.0}{\rmdefault}{\mddefault}{\updefault}{\color[rgb]{0,0,0}$\left(\rho_2^\flat, \eta_2^\flat\right)$}%
}}}}
\put(2446,-1336){\makebox(0,0)[lb]{\smash{{\SetFigFont{5}{6.0}{\rmdefault}{\mddefault}{\updefault}{\color[rgb]{0,0,0}$\left(\rho_2^\flat, \eta_2^\flat\right)$}%
}}}}
\put(4366,-3391){\makebox(0,0)[lb]{\smash{{\SetFigFont{5}{6.0}{\rmdefault}{\mddefault}{\updefault}{\color[rgb]{0,0,0}$A_2^\ell$}%
}}}}
\put(4306,-1996){\makebox(0,0)[lb]{\smash{{\SetFigFont{5}{6.0}{\rmdefault}{\mddefault}{\updefault}{\color[rgb]{0,0,0}$A_1^\ell$}%
}}}}
\put(4246,-331){\makebox(0,0)[lb]{\smash{{\SetFigFont{5}{6.0}{\rmdefault}{\mddefault}{\updefault}{\color[rgb]{0,0,0}$A_2^\ell$}%
}}}}
\end{picture}%

%% file: rhoeta_tau1_Lemma4.pdftex_t
\begin{picture}(0,0)%
\includegraphics{rhoeta_tau1_Lemma4.pdf}%
\end{picture}%
%
%
\setlength{\unitlength}{908sp}%
\begingroup\makeatletter\ifx\SetFigFont\undefined%
\gdef\SetFigFont#1#2#3#4#5{%
  \reset@font\fontsize{#1}{#2pt}%
  \fontfamily{#3}\fontseries{#4}\fontshape{#5}%
  \selectfont}%
\fi\endgroup%
\begin{picture}(10374,8399)(889,-8398)
\put(1201,-556){\makebox(0,0)[lb]{\smash{{\SetFigFont{5}{6.0}{\rmdefault}{\mddefault}{\updefault}{\color[rgb]{0,0,0}$\eta$}%
}}}}
\put(3406,-6736){\makebox(0,0)[lb]{\smash{{\SetFigFont{5}{6.0}{\rmdefault}{\mddefault}{\updefault}{\color[rgb]{0,0,0}$F$}%
}}}}
\put(10966,-8056){\makebox(0,0)[lb]{\smash{{\SetFigFont{5}{6.0}{\rmdefault}{\mddefault}{\updefault}{\color[rgb]{0,0,0}$\rho$}%
}}}}
\put(9436,-8146){\makebox(0,0)[lb]{\smash{{\SetFigFont{5}{6.0}{\rmdefault}{\mddefault}{\updefault}{\color[rgb]{0,0,0}$R$}%
}}}}
\put(8881,-4681){\makebox(0,0)[lb]{\smash{{\SetFigFont{5}{6.0}{\rmdefault}{\mddefault}{\updefault}{\color[rgb]{0,0,0}$C$}%
}}}}
\put(10216,-1921){\makebox(0,0)[lb]{\smash{{\SetFigFont{5}{6.0}{\rmdefault}{\mddefault}{\updefault}{\color[rgb]{0,0,0}$(\bar \rho_{1}, \bar \eta_{1})$}%
}}}}
\put(931,-8056){\makebox(0,0)[lb]{\smash{{\SetFigFont{5}{6.0}{\rmdefault}{\mddefault}{\updefault}{\color[rgb]{0,0,0}$0$}%
}}}}
\put(3766,-6121){\makebox(0,0)[lb]{\smash{{\SetFigFont{5}{6.0}{\rmdefault}{\mddefault}{\updefault}{\color[rgb]{0,0,0}$\left(\bar \rho_3, \bar \eta_3\right)$}%
}}}}
\put(6991,-4141){\makebox(0,0)[lb]{\smash{{\SetFigFont{5}{6.0}{\rmdefault}{\mddefault}{\updefault}{\color[rgb]{0,0,0}$\left(\rho_1^\flat, \eta_1^\flat\right)$}%
}}}}
\end{picture}%

%% file: rhoeta_tau11_Lemma4.pdftex_t
\begin{picture}(0,0)%
\includegraphics{rhoeta_tau11_Lemma4.pdf}%
\end{picture}%
%
%
\setlength{\unitlength}{908sp}%
\begingroup\makeatletter\ifx\SetFigFont\undefined%
\gdef\SetFigFont#1#2#3#4#5{%
  \reset@font\fontsize{#1}{#2pt}%
  \fontfamily{#3}\fontseries{#4}\fontshape{#5}%
  \selectfont}%
\fi\endgroup%
\begin{picture}(11445,8366)(136,-8190)
\put(1936,-8011){\makebox(0,0)[lb]{\smash{{\SetFigFont{5}{6.0}{\rmdefault}{\mddefault}{\updefault}{\color[rgb]{0,0,0}$(\bar \rho_{1}, \bar \eta_{1})$}%
}}}}
\put(4966,-76){\makebox(0,0)[lb]{\smash{{\SetFigFont{5}{6.0}{\rmdefault}{\mddefault}{\updefault}{\color[rgb]{0,0,0}$t$}%
}}}}
\put(10786,-8116){\makebox(0,0)[lb]{\smash{{\SetFigFont{5}{6.0}{\rmdefault}{\mddefault}{\updefault}{\color[rgb]{0,0,0}$x$}%
}}}}
\put(1996,-4816){\makebox(0,0)[lb]{\smash{{\SetFigFont{5}{6.0}{\rmdefault}{\mddefault}{\updefault}{\color[rgb]{0,0,0}$(\bar \rho_{1}, \bar \eta_{1})$}%
}}}}
\put(5101,-8026){\makebox(0,0)[lb]{\smash{{\SetFigFont{5}{6.0}{\rmdefault}{\mddefault}{\updefault}{\color[rgb]{0,0,0}$0$}%
}}}}
\put(8626,-6331){\makebox(0,0)[lb]{\smash{{\SetFigFont{5}{6.0}{\rmdefault}{\mddefault}{\updefault}{\color[rgb]{0,0,0}$\left(\bar \rho_3, \bar \eta_3\right)$}%
}}}}
\put(7471,-8041){\makebox(0,0)[lb]{\smash{{\SetFigFont{5}{6.0}{\rmdefault}{\mddefault}{\updefault}{\color[rgb]{0,0,0}$\left(\bar \rho_3, \bar \eta_3\right)$}%
}}}}
\put(151,-7501){\makebox(0,0)[lb]{\smash{{\SetFigFont{5}{6.0}{\rmdefault}{\mddefault}{\updefault}{\color[rgb]{1,0,0}$I_1$}%
}}}}
\put(11566,-7576){\makebox(0,0)[lb]{\smash{{\SetFigFont{5}{6.0}{\rmdefault}{\mddefault}{\updefault}{\color[rgb]{1,0,0}$I_3$}%
}}}}
\put(4591,-3061){\makebox(0,0)[lb]{\smash{{\SetFigFont{5}{6.0}{\rmdefault}{\mddefault}{\updefault}{\color[rgb]{0,0,0}$\tau_{1}^\ell$}%
}}}}
\put(6586,-3091){\makebox(0,0)[lb]{\smash{{\SetFigFont{5}{6.0}{\rmdefault}{\mddefault}{\updefault}{\color[rgb]{0,0,0}$\left(\rho_1^\flat, \eta_1^\flat\right)$}%
}}}}
\end{picture}%

%% file: rhoeta_tau111_Lemma4.pdftex_t
\begin{picture}(0,0)%
\includegraphics{rhoeta_tau111_Lemma4.pdf}%
\end{picture}%
%
%
\setlength{\unitlength}{829sp}%
\begingroup\makeatletter\ifx\SetFigFont\undefined%
\gdef\SetFigFont#1#2#3#4#5{%
  \reset@font\fontsize{#1}{#2pt}%
  \fontfamily{#3}\fontseries{#4}\fontshape{#5}%
  \selectfont}%
\fi\endgroup%
\begin{picture}(10377,8399)(886,-8398)
\put(1201,-556){\makebox(0,0)[lb]{\smash{{\SetFigFont{5}{6.0}{\rmdefault}{\mddefault}{\updefault}{\color[rgb]{0,0,0}$\eta$}%
}}}}
\put(5071,-5926){\makebox(0,0)[lb]{\smash{{\SetFigFont{5}{6.0}{\rmdefault}{\mddefault}{\updefault}{\color[rgb]{0,0,0}$F$}%
}}}}
\put(10966,-8056){\makebox(0,0)[lb]{\smash{{\SetFigFont{5}{6.0}{\rmdefault}{\mddefault}{\updefault}{\color[rgb]{0,0,0}$\rho$}%
}}}}
\put(9436,-8146){\makebox(0,0)[lb]{\smash{{\SetFigFont{5}{6.0}{\rmdefault}{\mddefault}{\updefault}{\color[rgb]{0,0,0}$R$}%
}}}}
\put(8506,-2461){\makebox(0,0)[lb]{\smash{{\SetFigFont{5}{6.0}{\rmdefault}{\mddefault}{\updefault}{\color[rgb]{0,0,0}$C$}%
}}}}
\put(901,-8131){\makebox(0,0)[lb]{\smash{{\SetFigFont{5}{6.0}{\rmdefault}{\mddefault}{\updefault}{\color[rgb]{0,0,0}$0$}%
}}}}
\put(7501,-3991){\makebox(0,0)[lb]{\smash{{\SetFigFont{5}{6.0}{\rmdefault}{\mddefault}{\updefault}{\color[rgb]{0,0,0}$(\bar \rho_{2}, \bar \eta_{2})$}%
}}}}
\put(10246,-4126){\makebox(0,0)[lb]{\smash{{\SetFigFont{5}{6.0}{\rmdefault}{\mddefault}{\updefault}{\color[rgb]{0,0,0}$\left(R,R \bar w_2\right)$}%
}}}}
\end{picture}%

%% file: rhoeta_tau2_Lemma4.pdftex_t
\begin{picture}(0,0)%
\includegraphics{rhoeta_tau2_Lemma4.pdf}%
\end{picture}%
%
%
\setlength{\unitlength}{908sp}%
\begingroup\makeatletter\ifx\SetFigFont\undefined%
\gdef\SetFigFont#1#2#3#4#5{%
  \reset@font\fontsize{#1}{#2pt}%
  \fontfamily{#3}\fontseries{#4}\fontshape{#5}%
  \selectfont}%
\fi\endgroup%
\begin{picture}(10377,8399)(886,-8398)
\put(1201,-556){\makebox(0,0)[lb]{\smash{{\SetFigFont{5}{6.0}{\rmdefault}{\mddefault}{\updefault}{\color[rgb]{0,0,0}$\eta$}%
}}}}
\put(4516,-5551){\makebox(0,0)[lb]{\smash{{\SetFigFont{5}{6.0}{\rmdefault}{\mddefault}{\updefault}{\color[rgb]{0,0,0}$F$}%
}}}}
\put(10966,-8056){\makebox(0,0)[lb]{\smash{{\SetFigFont{5}{6.0}{\rmdefault}{\mddefault}{\updefault}{\color[rgb]{0,0,0}$\rho$}%
}}}}
\put(9436,-8146){\makebox(0,0)[lb]{\smash{{\SetFigFont{5}{6.0}{\rmdefault}{\mddefault}{\updefault}{\color[rgb]{0,0,0}$R$}%
}}}}
\put(8671,-2596){\makebox(0,0)[lb]{\smash{{\SetFigFont{5}{6.0}{\rmdefault}{\mddefault}{\updefault}{\color[rgb]{0,0,0}$C$}%
}}}}
\put(901,-8131){\makebox(0,0)[lb]{\smash{{\SetFigFont{5}{6.0}{\rmdefault}{\mddefault}{\updefault}{\color[rgb]{0,0,0}$0$}%
}}}}
\put(10216,-4396){\makebox(0,0)[lb]{\smash{{\SetFigFont{5}{6.0}{\rmdefault}{\mddefault}{\updefault}{\color[rgb]{0,0,0}$(R, R \bar w_2)$}%
}}}}
\put(5686,-6616){\makebox(0,0)[lb]{\smash{{\SetFigFont{5}{6.0}{\rmdefault}{\mddefault}{\updefault}{\color[rgb]{0,0,0}$\left(\rho_2^\flat, \eta_2^\flat\right)$}%
}}}}
\put(7141,-3691){\makebox(0,0)[lb]{\smash{{\SetFigFont{5}{6.0}{\rmdefault}{\mddefault}{\updefault}{\color[rgb]{0,0,0}$\left(\rho_1^\flat, \eta_1^\flat\right)$}%
}}}}
\end{picture}%

%% file: rhoeta_tau22_Lemma4.pdftex_t
\begin{picture}(0,0)%
\includegraphics{rhoeta_tau22_Lemma4.pdf}%
\end{picture}%
%
%
\setlength{\unitlength}{1066sp}%
\begingroup\makeatletter\ifx\SetFigFont\undefined%
\gdef\SetFigFont#1#2#3#4#5{%
  \reset@font\fontsize{#1}{#2pt}%
  \fontfamily{#3}\fontseries{#4}\fontshape{#5}%
  \selectfont}%
\fi\endgroup%
\begin{picture}(11340,8154)(211,-8248)
\put(4966,-1261){\makebox(0,0)[lb]{\smash{{\SetFigFont{5}{6.0}{\rmdefault}{\mddefault}{\updefault}{\color[rgb]{0,0,0}$t$}%
}}}}
\put(10921,-8041){\makebox(0,0)[lb]{\smash{{\SetFigFont{5}{6.0}{\rmdefault}{\mddefault}{\updefault}{\color[rgb]{0,0,0}$x$}%
}}}}
\put(1066,-7921){\makebox(0,0)[lb]{\smash{{\SetFigFont{5}{6.0}{\rmdefault}{\mddefault}{\updefault}{\color[rgb]{0,0,0}$(\bar \rho_{2}, \bar \eta_{2})$}%
}}}}
\put(2581,-6166){\makebox(0,0)[lb]{\smash{{\SetFigFont{5}{6.0}{\rmdefault}{\mddefault}{\updefault}{\color[rgb]{0,0,0}$(R, R \bar w_2)$}%
}}}}
\put(3061,-7951){\makebox(0,0)[lb]{\smash{{\SetFigFont{5}{6.0}{\rmdefault}{\mddefault}{\updefault}{\color[rgb]{0,0,0}$(R, R \bar w_2)$}%
}}}}
\put(226,-7531){\makebox(0,0)[lb]{\smash{{\SetFigFont{5}{6.0}{\rmdefault}{\mddefault}{\updefault}{\color[rgb]{1,0,0}$I_2$}%
}}}}
\put(11536,-7561){\makebox(0,0)[lb]{\smash{{\SetFigFont{5}{6.0}{\rmdefault}{\mddefault}{\updefault}{\color[rgb]{1,0,0}$I_3$}%
}}}}
\put(8641,-8011){\makebox(0,0)[lb]{\smash{{\SetFigFont{5}{6.0}{\rmdefault}{\mddefault}{\updefault}{\color[rgb]{0,0,0}$(\bar \rho_{3}, \bar \eta_{3})$}%
}}}}
\put(5941,-4171){\makebox(0,0)[lb]{\smash{{\SetFigFont{5}{6.0}{\rmdefault}{\mddefault}{\updefault}{\color[rgb]{0,0,0}$\left(\rho_2^\flat, \eta_2^\flat\right)$}%
}}}}
\put(5776,-7936){\makebox(0,0)[lb]{\smash{{\SetFigFont{5}{6.0}{\rmdefault}{\mddefault}{\updefault}{\color[rgb]{0,0,0}$\tau_{1}^\ell$}%
}}}}
\put(6121,-3136){\makebox(0,0)[lb]{\smash{{\SetFigFont{5}{6.0}{\rmdefault}{\mddefault}{\updefault}{\color[rgb]{0,0,0}$\tau_{1}^\ell+\tau_{2}^\ell$}%
}}}}
\put(7081,-6496){\makebox(0,0)[lb]{\smash{{\SetFigFont{5}{6.0}{\rmdefault}{\mddefault}{\updefault}{\color[rgb]{0,0,0}$\left(\rho_1^\flat, \eta_1^\flat\right)$}%
}}}}
\put(6781,-7966){\makebox(0,0)[lb]{\smash{{\SetFigFont{5}{6.0}{\rmdefault}{\mddefault}{\updefault}{\color[rgb]{0,0,0}$\left(\rho_1^\flat, \eta_1^\flat\right)$}%
}}}}
\end{picture}%

%% file: rhoeta_tau222_Lemma4.pdftex_t
\begin{picture}(0,0)%
\includegraphics{rhoeta_tau222_Lemma4.pdf}%
\end{picture}%
%
%
\setlength{\unitlength}{908sp}%
\begingroup\makeatletter\ifx\SetFigFont\undefined%
\gdef\SetFigFont#1#2#3#4#5{%
  \reset@font\fontsize{#1}{#2pt}%
  \fontfamily{#3}\fontseries{#4}\fontshape{#5}%
  \selectfont}%
\fi\endgroup%
\begin{picture}(10377,8399)(886,-8398)
\put(1201,-556){\makebox(0,0)[lb]{\smash{{\SetFigFont{5}{6.0}{\rmdefault}{\mddefault}{\updefault}{\color[rgb]{0,0,0}$\eta$}%
}}}}
\put(5536,-5131){\makebox(0,0)[lb]{\smash{{\SetFigFont{5}{6.0}{\rmdefault}{\mddefault}{\updefault}{\color[rgb]{0,0,0}$F$}%
}}}}
\put(10966,-8056){\makebox(0,0)[lb]{\smash{{\SetFigFont{5}{6.0}{\rmdefault}{\mddefault}{\updefault}{\color[rgb]{0,0,0}$\rho$}%
}}}}
\put(9436,-8146){\makebox(0,0)[lb]{\smash{{\SetFigFont{5}{6.0}{\rmdefault}{\mddefault}{\updefault}{\color[rgb]{0,0,0}$R$}%
}}}}
\put(6676,-5566){\makebox(0,0)[lb]{\smash{{\SetFigFont{5}{6.0}{\rmdefault}{\mddefault}{\updefault}{\color[rgb]{0,0,0}$C$}%
}}}}
\put(901,-8131){\makebox(0,0)[lb]{\smash{{\SetFigFont{5}{6.0}{\rmdefault}{\mddefault}{\updefault}{\color[rgb]{0,0,0}$0$}%
}}}}
\put(10306,-1486){\makebox(0,0)[lb]{\smash{{\SetFigFont{5}{6.0}{\rmdefault}{\mddefault}{\updefault}{\color[rgb]{0,0,0}$\left(R,R \bar w_1\right)$}%
}}}}
\put(6916,-4216){\makebox(0,0)[lb]{\smash{{\SetFigFont{5}{6.0}{\rmdefault}{\mddefault}{\updefault}{\color[rgb]{0,0,0}$\left(\rho_1^\flat, \eta_1^\flat\right)$}%
}}}}
\end{picture}%

%% file: rhoeta_tau2222_Lemma4.pdftex_t
\begin{picture}(0,0)%
\includegraphics{rhoeta_tau2222_Lemma4.pdf}%
\end{picture}%
%
%
\setlength{\unitlength}{1145sp}%
\begingroup\makeatletter\ifx\SetFigFont\undefined%
\gdef\SetFigFont#1#2#3#4#5{%
  \reset@font\fontsize{#1}{#2pt}%
  \fontfamily{#3}\fontseries{#4}\fontshape{#5}%
  \selectfont}%
\fi\endgroup%
\begin{picture}(6165,8109)(163,-8173)
\put(6166,-1741){\makebox(0,0)[lb]{\smash{{\SetFigFont{5}{6.0}{\rmdefault}{\mddefault}{\updefault}{\color[rgb]{0,0,0}$t$}%
}}}}
\put(676,-7996){\makebox(0,0)[lb]{\smash{{\SetFigFont{5}{6.0}{\rmdefault}{\mddefault}{\updefault}{\color[rgb]{0,0,0}$x$}%
}}}}
\put(361,-5746){\makebox(0,0)[lb]{\smash{{\SetFigFont{5}{6.0}{\rmdefault}{\mddefault}{\updefault}{\color[rgb]{0,0,0}$(\bar \rho_{1}, \bar \eta_{1})$}%
}}}}
\put(1531,-7996){\makebox(0,0)[lb]{\smash{{\SetFigFont{5}{6.0}{\rmdefault}{\mddefault}{\updefault}{\color[rgb]{0,0,0}$(\bar \rho_{1}, \bar \eta_{1})$}%
}}}}
\put(3556,-2656){\makebox(0,0)[lb]{\smash{{\SetFigFont{5}{6.0}{\rmdefault}{\mddefault}{\updefault}{\color[rgb]{0,0,0}$\left(R,R \bar w_1\right)$}%
}}}}
\put(211,-7516){\makebox(0,0)[lb]{\smash{{\SetFigFont{5}{6.0}{\rmdefault}{\mddefault}{\updefault}{\color[rgb]{1,0,0}$I_1$}%
}}}}
\put(6151,-3061){\makebox(0,0)[lb]{\smash{{\SetFigFont{5}{6.0}{\rmdefault}{\mddefault}{\updefault}{\color[rgb]{0,0,0}$\tau_{1}^\ell+\tau_{2}^\ell$}%
}}}}
\put(6091,-7981){\makebox(0,0)[lb]{\smash{{\SetFigFont{5}{6.0}{\rmdefault}{\mddefault}{\updefault}{\color[rgb]{0,0,0}$\tau_{1}^\ell$}%
}}}}
\put(3451,-6796){\makebox(0,0)[lb]{\smash{{\SetFigFont{5}{6.0}{\rmdefault}{\mddefault}{\updefault}{\color[rgb]{0,0,0}$\left(\rho_1^\flat, \eta_1^\flat\right)$}%
}}}}
\end{picture}%